\documentclass[12pt,sort&compress]{elsarticle}

\makeatletter
\def\ps@pprintTitle{\let\@oddhead\@empty
  \let\@evenhead\@empty
  \def\@oddfoot{\reset@font\hfil\thepage\hfil}
  \let\@evenfoot\@oddfoot
}
\makeatother

\usepackage{amsmath}
\usepackage{amsfonts}
\usepackage{amssymb}
\usepackage{amsfonts}
\usepackage{amsthm}
\usepackage{amscd}
\usepackage{mathtools}
\usepackage{commath}
\usepackage{lmodern}
\usepackage{color}
\usepackage{a4wide}
\usepackage{bm} \usepackage{todonotes}

\graphicspath{ {img/} }

\usepackage[draft]{fixme}
\fxsetup{theme=color}
\usepackage{hyperref}

\usepackage[czech,english]{babel}
\usepackage[IL2]{fontenc}
\usepackage[utf8]{inputenc}

\newtheorem{theorem}{Theorem}[section]
\newtheorem{lemma}[theorem]{Lemma}

\theoremstyle{definition}
\newtheorem{definition}[theorem]
{Definition}
\theoremstyle{remark}
\newtheorem{remark}[theorem]{\upshape\bfseries Remark}
\newtheorem{example}[theorem]{\upshape\bfseries Example}

\newcommand{\C}{\mathbb{C}}
\newcommand{\R}{\mathbb{R}}
\renewcommand{\H}{\mathbb{H}}

\newcommand*{\qi}{\mathbf{i}}
\newcommand*{\qj}{\mathbf{j}}
\newcommand*{\qk}{\mathbf{k}}

\newcommand*{\Cj}[1]{{#1}^\ast}
\renewcommand{\vec}[1]{\mathbf{#1}}
\DeclareMathOperator{\Span}{Span}
\DeclareMathOperator{\RE}{Re}
\DeclareMathOperator{\IM}{Im}
\DeclareMathOperator{\lcm}{lcm}

\begin{document}

\begin{frontmatter}

\title{Optimal interpolation with spatial rational Pythagorean hodograph curves}

\author[1]{Hans-Peter Schröcker}
\ead{hans-peter.schroecker@uibk.ac.at}
\address[1]{Universität Innsbruck, Department of Basic Sciences in Engineering Sciences, Technikerstr.~13, 6020 Innsbruck, Austria}

\author[2]{Zbyněk Šír}
\ead{zbynek.sir@karlin.mff.cuni.cz}
\address[2]{Charles University, Faculty of Mathematics and Physic, Sokolovská 83, Prague 186 75, Czech Republic}

\begin{abstract}
    Using a residuum approach, we provide a complete description of the space of the rational spatial curves of given tangent directions. The rational Pythagorean hodograph curves are obtained as a special case when the norm of the direction field is a perfect square. The basis for the curve space is given explicitly. Consequently a number of interpolation problems ($G^1$, $C^1$, $C^2$, $C^1/G^2$) in this space become linear, cusp avoidance can be encoded by linear inequalities, and optimization problems like  minimal energy or optimal length are quadratic and can be solved efficiently via quadratic programming. We outline the interpolation/optimization strategy and demonstrate it on several examples.
 \end{abstract}


\end{frontmatter}

\section{Introduction}
\label{sec:introduction}

Pythagorean hodograph (PH) curves (see \cite{farouki08} and the references cited
therein), form a remarkable subclass of polynomial parametric curves. They have
a piecewise polynomial arc length function and, in the planar case, rational
offset curves. These curves provide an elegant solution of various difficult
problems occurring in applications, in particular in the context of CNC
(computer-numerical-control), see e.g. \cite{Farouki2019}.

Spatial PH curves were introduced by Farouki and Sakkalis in 1994 \cite{farouki94a}, and they have later been characterized using results about Pythagorean quadruples in the ring of polynomials and quaternion calculus \cite{farouki02a,Dietz,choi02b}. Spatial PH curves can be equipped with rational frames, which, in certain cases, are rotation minimizing \cite{RMF1}.

Due to the constrained nature of PH curves, constructions which are linear in the case of polynomial curves typically become \emph{nonlinear} in the PH case and produce infinitely many solutions in the spatial case. Despite this fact various constructions were given in past publications. Hermite interpolation of $G^1$ boundary data was addressed in \cite{Juettler:99b}, and $C^1$ Hermite interpolation by PH quintics was discussed in \cite{FaroukiSpace2}. In the latter case, the authors identify a family of interpolants to any $C^1$ Hermite data which depends on two free parameters, and a heuristic choice for them is given. Later, this construction has also been related to helical interpolants \cite{PHhelices}. Methods for choosing an optimal interpolant were presented e.g. in \cite{sir05,WOS:000526979400002,WOS:000582397100008}. The $C^2$ interpolation problem was solved by PH curves of degree $9$ in \cite{SirC2} and by quintic splines in \cite{C2length,GianelliSplines}.

Rational spatial PH curves were first investigated in \cite{FaroukiSir} where they were described as edges of regression of developable surfaces enveloped by special system of tangent planes. The dual approach was continued in series of papers \cite{krajnc1,krajnc2,krajnc3} where the authors study low degree and low class curves. The theory of rotation minimizing frames was transferred from polynomial to rational spatial curves in~\cite{FaroukiSir2}.

In these  publications on rational PH curves various interpolation problems are mentioned, see e.g. \cite{krajnc3,FaroukiSir2}. There, the interpolation problems present both linear (due to the envelope formula) and non-linear aspects (due to the PH nature of the curves). It is one of the aims of this paper to stress the linear aspect of PH interpolation problems. This is possible because all rational curves with a common tangent field form a linear space. Representations used so far did not really allow exploit this linearity to solve the interpolation and approximation problems efficiently. The dual representation is linear but does not provide an efficient control over subspaces, degrees, numerator-denominator cancellations or the occurrence of cusps.

For this reason in two recent papers \cite{kalkan22,schroecker22} a new approach
was developed to describe a natural basis of the linear space of PH curves with
given tangent field. In \cite{kalkan22}, the theory of rational PH curves was
developed from so called framing rational motions. Representing the spherical
part of the motion by a polynomial $\mathcal{A}(t) \in \H[t]$ in the ring
$\H[t]$ of polynomials with quaternion coefficients and fixing the PH curve's
denominator its numerator was computed by solving a modestly sized and
well-structured system of linear equations. As a main result the understanding
of truly rational solutions within generically polynomial ones was obtained. In
\cite{schroecker22} the complete structure of the solution spaces was obtained
for the generic case.

In the present paper we extend the previous ideas to a residuum based approach,
cf. \cite{FaroukiSakkalis2019}. Basically, rational PH curves are obtained by
integrating expressions of type $\lambda(t) \mathbf F(t)$ where $\lambda(t)$ is
a rational function and $\mathbf F(t) = \mathcal A(t) \qi \Cj{\mathcal A}(t) \in
\mathbb H[t]$ is the typical vector valued polynomial that is also used for the
construction of polynomial PH curves, \cite{farouki08}. It is given in terms of
a quaternion polynomial $\mathcal A(t) \in \mathbb H[t]$, its quaternion
conjugate $\Cj{\mathcal A}(t)$, and the quaternion unit $\qi$ with the property
$\qi^2 = -1$. The basic idea is to impose conditions on $\lambda(t)$ that ensure
zero residua in the Laurent expansions at any root $\beta \in \C$ of the denominator of $\lambda(t)$. This guarantees rational integrands and generalizes a well-known construction of polynomial PH curves to the rational case.

In Sections~\ref{sec:residuum} and \ref{sec:non-regular} we explain these ideas in detail and obtain a general and complete construction for rational PH curves. We also provide an explicit description of the basis for spaces of PH curves of given direction field. Consequently we obtain a tool to solve efficiently various interpolation and optimization problems. It is described in Section~\ref{sec:interpolation}. In Section~\ref{sec:examples} we discuss several examples.

\section{Curves with Given Tangent Indicatrix}
\label{sec:residuum}

We are now going to solve the following problem: Given a rational vector valued quaternion polynomial $\mathbf F(t)\in \mathbb H[t]$ (which we identify with a vector valued polynomial in $\mathbb R^3[t]$), determine the vector space $\mathcal R$ of all the spatial rational curves $\mathbf{r}(t)$ having $\mathbf{F}(t) $ for its tangent field, i.e. satisfying
\begin{equation}
  \label{eq:1}
  \mathbf r'(t)\times {\mathbf F}(t)=0.
\end{equation}

In \cite{kalkan22,schroecker22,FaroukiSir,krajnc1,FaroukiSir2} this problem was studied for $\mathbf F(t) \coloneqq \mathcal{A}(t) \qi \Cj{\mathcal{A}}(t)$. In particular in \cite{schroecker22} extensive results were obtained for generic cases. The method we propose here is more general and straightforward. It works for arbitrary $\mathbf F(t)$, gives explicit insight into the structure of $\mathcal R$, and handles both, generic and non-generic cases.

As already explained in \cite{kalkan22,schroecker22}, it is no loss of generality to assume that the field $\mathbf F(t)$ is \emph{polynomial} and \emph{without real polynomial factors}. Indeed, multiplying any given rational field by the common denominator and factoring out any common factor of the numerator does not influence condition \eqref{eq:1} and therefore the space $\mathcal R$ will remain the same.

We will assume in the remainder of this text that $\mathbf F(t)$ is a truly
spatial\footnote{Generalizations of our theory to fields in $\mathbb R^N$ with
  $N > 3$ or specializations to fields in $\mathbb R^2$ are straightforward. If
  the tangent field $\mathbf F(t)$ is parallel to a vector subspace, then so is
  the resulting curve $\mathbf r(t)$. Thus picking $\mathcal A(t)$ (for example)
  with coefficients in the sub-algebra spanned by $1$ and $\qk$ will result in
  planar rational PH curves.} polynomial field of degree $n$ without polynomial
factors. Consequently for any $\beta \in \mathbb C$ we have the expansion
\begin{equation}
  \label{eq:2}
  \mathbf F(t)=\sum_{j=0}^n \mathbf f_{j}(t-\beta)^{j},
\end{equation}
with $\mathbf f_{n} \neq \mathbf 0$ and $\Span\{\mathbf f_{0}, \mathbf f_{1},\ldots, \mathbf f_{n} \}=\mathbb R^3$. The coefficient $\mathbf f_{0}$ cannot vanish as $(t-\beta)$ is not a factor of $\mathbf F(t)$ by assumption.

At first sight the proposed problem seems to be easy. Any rational $\mathbf r(t)$ satisfying \eqref{eq:1} and therefore tangent to $\mathbf F(t)$ can be expressed in the form
\begin{equation}
  \label{eq:3}
  \mathbf r(t)=\int \lambda(t)\mathbf F(t) \;\mathrm{d}t + \mathbf c,
\end{equation}
where $\lambda(t)\in \mathbb R(t)$ is a rational function and $\mathbf c \in \mathbb R^3$. Because of $\mathbf r'(t) = \lambda(t) \mathbf F (t)$, the rational function $\lambda(t)$ accounts for the differences in parametric speed among all rational curves with tangent field parallel to $\mathbf F(t)$. We therefore refer to it as the \emph{speed function} (even if it is not the speed function known from differential geometry).

The issue with formula \eqref{eq:3} is due to the fact that not all $\lambda(t)$ provide rational curves as integration may produce irrational functions (logarithms and arctangents). This lead to alternative approaches, namely the dual (envelope based) approach of \cite{FaroukiSir,krajnc1,FaroukiSir2} and the kinematics based approach of~\cite{kalkan22,schroecker22}.

Nonetheless, the functions $\lambda(t)$ providing rational curves form a linear subspace $\mathcal L \subset \mathbb R(t)$ of rational functions. There is a simple way to interpret formula \eqref{eq:3} as a bijective linear mapping from the Cartesian product of $\mathcal L$ and $\mathbb R^3$ onto the vector space $\mathcal R$ that is $\mathcal L \times \mathbb R^3 \to \mathcal R$, $(\lambda(t), \mathbf c) \mapsto \mathbf r(t)$.\footnote{In order to make this a well-defined map, it is necessary to normalize (fix one representative) of the expression $\int \lambda(t)\mathbf F(t) \;\mathrm{d}t$. One possibility for doing so is to require that $\mathbf F(\beta) = 0$ for some given $\beta \in \mathbb R$. Another way is to use the partial fraction decomposition (see infra) and require that all the integrals for individual roots have zero absolute term.}

In order to determine the rational functions $\lambda(t) \in \mathcal L$ let us consider
\begin{lemma}[Partial fraction decomposition for rational functions]
  \label{lem:pfd}
  If $\lambda(t) \in \mathbb R(t)$ is a rational function whose denominator in reduced form can be written as $\prod_{i=1}^m (t-\beta_i)^{k_i}$ with $k_i \in \mathbb N$ and pairwise different values $\beta_1$, $\beta_2$, \ldots, $\beta_m \in \mathbb C$, there exists a unique polynomial $p(t) \in \R[t]$ and coefficients $\lambda_{i,j} \in \mathbb C$ such that
  \begin{equation}
    \label{eq:4}
    \lambda(t)=p(t)+\sum_{i=1}^m \sum_{j=-k_i}^{-1} \lambda_{i,j}(t-\beta_i)^j.
  \end{equation}
\end{lemma}
It will become clear in the remainder of this article why we prefer a formulation of this well-known result (see e.g. \cite{Pathak2015-gk}) in terms of the complex zeros $\beta_i$ and hence likewise complex coefficients $\lambda_{i,j}$. Suitably combining pairs of conjugate complex summands will result in the real partial fraction decomposition.

Combining \eqref{eq:2} and \eqref{eq:4} we see that only coefficients $\lambda_{i,j}$ with $(-n-1)\leq j\leq -1$ influence the residua with respect to $\beta_i$ of the integrand in \eqref{eq:3} and therefore rationality of $\mathbf r(t)$. This leads us to:

\begin{definition}
  \label{def:spaces}
  For given tangent field $\mathbf F(t) \in \mathbb R^3[t]$ of degree $n$, let us define the following linear subspaces of $\mathcal R$:
  \begin{enumerate}
  \item The space $\mathcal P$ containing all \emph{polynomial} PH curves. Its elements are obtained from \eqref{eq:3} for polynomial $\lambda(t) \in \mathbb R[t]$. It is of infinite (but countable) dimension and contains the three-dimensional subspace $\mathbb R^3$ of constant PH curves.\footnote{In literature on PH curves, these constant curves are often disregarded but in our context it is very natural to consider them as valid solutions to \eqref{eq:1}. Their presence accounts for the translation invariance of $\mathcal R$ and~$\mathcal P$.} A basis of $\mathcal P$ is given by taking a basis of $\mathbb R^3$ and adding all curves obtained from $\lambda(t) = (t-\beta)^p$ for $p \ge 0$ and some fixed $\beta \in \mathbb R$.
  \item The one-dimensional space $\mathcal R^{\beta,r}$ of \emph{regular rational PH curves,} spanned by the rational PH curve obtained from \eqref{eq:3} for $\lambda(t) = (t - \beta)^r$, $\mathbf c = 0$, $\beta \in \mathbb C$, and $r < -n-1$.
  \item Spaces of \emph{non-regular rational PH curves} $\mathcal N^{\beta}$ obtained from \eqref{eq:3} for
    \begin{equation}
      \label{eq:5}
      \lambda(t)=\sum_{j=-n-1}^{-1} \lambda_{j}(t-\beta)^j,
    \end{equation}
    where $\mathbf c = 0$, $\beta \in \mathbb C$ and $\lambda_{-n-1}$, \ldots, $\lambda_{-1} \in \C$ are subject to the zero-residuum constraint $\lambda_{-1}\mathbf f_{0}+\lambda_{-2}\mathbf f_{1}+\lambda_{-3}\mathbf f_{2}+ \cdots +\lambda_{-n-1}\mathbf f_{n}=0$ that ensures rationality of the integral \eqref{eq:3}.
\end{enumerate}
\end{definition}

As a direct consequence of Lemma~\ref{lem:pfd} and Definition~\ref{def:spaces} we obtain the following description of the full solution space $\mathcal R$
\begin{theorem}
  The space $\mathcal R$ of rational curves of given polynomial direction field allows the decomposition
  \begin{equation}
    \label{eq:6}
    \mathcal {R}=  \mathcal P \oplus \bigoplus_{\beta \in \mathbb C}\left (\mathcal N^{\beta} \oplus \bigoplus_{j=-\infty}^{-n-1} {\mathcal R}^{\beta,j}\right ).
   \end{equation}
   as direct sum of polynomial solutions as well as regular and non-regular solutions.
\end{theorem}

\section{Spaces of Non-Regular Curves}
\label{sec:non-regular}

The structure (dimension, bases, subspaces) of the space $\mathcal{P}$ of polynomial PH curves and the space $\mathcal{R}^{\beta,m}$ of regular rational PH curves is easy to understand. In this section we study the space $\mathcal N^\beta$ of non-regular rational PH curves in more detail. The non-regular solutions are important not only because they complete the structure of $\mathcal R$, but also because they contain the truly rational (non-polynomial) curves of minimal degree \cite{krajnc1,schroecker22}.

Let us fix a root $\beta \in \mathbb C$ and denote by $\mathcal L^\beta\subset \mathcal L$ the linear space of rational functions having at most one singularity $t=\beta$. Our detailed analysis in \cite{schroecker22} can be interpreted as the study of the image of $\mathcal L^\beta$ space under the mapping \eqref{eq:3}, i.e. of the space
\begin{equation*}
  \mathcal {R^{\beta}} \coloneqq  \mathcal P \oplus \mathcal N^{\beta} \oplus \bigoplus_{j=-\infty}^{-n-1} \mathcal R^{\beta,j}
\end{equation*}
for a fixed $\beta$ which is \emph{generic}. Genericity in this context means that any three consecutive coefficients $\mathbf f_{j-1}$, $\mathbf f_{j}$, $\mathbf f_{j+1}$ in the Taylor expansion \eqref{eq:2} of $\mathbf F(t)$ are linearly independent. Consequently we obtained a good understanding of the subspace $\mathcal{N}^{\beta}$ for the generic cases. Its structure in non-generic situations remained, however, unclear, even if an algorithm for computing a basis for any concrete $\beta$ was outlined in \cite{schroecker22}. With the approach in this paper we are able to describe both the generic and non-generic case in a unified way.

\begin{lemma}
  \label{lem:dimN}
  For any $\beta \in \mathbb C$ we have
  \begin{equation*}
    \dim \mathcal N^{\beta} = n - 2.
  \end{equation*}
\end{lemma}

\begin{proof}
  For any fixed $\beta \in \mathbb C$ the space of rational functions of the form \eqref{eq:5} has dimension $n+1$. With $\mathbf F(t) = \sum_{j=0}^n \mathbf f_j(t-\beta)^j$, the zero-residuum condition reads
  \begin{equation}
    \label{eq:7}
    \lambda_{-1}\mathbf f_{0}+\lambda_{-2}\mathbf f_{1}+\lambda_{-3}\mathbf f_{2}+ \cdots +\lambda_{-n-1}\mathbf f_{n}=0.
  \end{equation}
  This is a homogeneous system of rank $3$ (because $\mathbf{F}(t)$ is assumed to be spatial) with $n+1$ variables. Its solution space is thus of dimension $n- 2$. This space is injectively mapped onto $\mathcal{N}^\beta$ via \eqref{eq:3} with~$\mathbf c = 0$.
\end{proof}

In order to obtain a basis of $\mathcal N^{\beta}$ we describe the solution space of \eqref{eq:7} and then map it linearly to the curve space. The most natural way is to choose three linearly independent coefficients $\mathbf f_{i_1}$, $\mathbf f_{i_2}$, $\mathbf f_{i_3}$ where $0 \le i_1$, $i_2$, $i_3 \le n$ and express the corresponding unknowns $\lambda_{(-1-i_1)}$, $\lambda_{(-1-i_2)}$, $\lambda_{(-1-i_3)}$ in terms of the remaining $\lambda_i$'s which are considered to be free parameters.

The curve $\mathbf r(t)$ given by \eqref{eq:3} with $\lambda(t)$ of the form \eqref{eq:5} is rational under the zero residuum condition \eqref{eq:7} and can then be expressed explicitly in a very symmetrical way as
\begin{equation*}
  \mathbf r(t)=\sum_{j=-n}^{n} \mathbf r_j(t-\beta)^j,
\end{equation*}
where
\begin{equation}\label{eq:8}
  \begin{aligned}
    \mathbf r_{-n}&= -\frac{1}{n}(\lambda_{-n-1}\mathbf f_0),\\
    \mathbf r_{-n+1}&= -\frac{1}{n-1}(\lambda_{-n}\mathbf f_0+\lambda_{-n-1}\mathbf f_1),\\
    \mathbf r_{-n+2}&= -\frac{1}{n-2}(\lambda_{-n+1}\mathbf f_0+\lambda_{-n}\mathbf f_1+\lambda_{-n-1}\mathbf f_2),\\
    \vdots & \\
    \mathbf r_{0}&= \mathbf 0,\\
    \vdots & \\
    \mathbf r_{n-2}&= \frac{1}{n-2}(\lambda_{-1}\mathbf f_{n-2}+\lambda_{-2}\mathbf f_{n-1}+\lambda_{-3}\mathbf f_{n}),\\
    \mathbf r_{n-1}&= \frac{1}{n-1}(\lambda_{-1}\mathbf f_{n-1}+\lambda_{-2}\mathbf f_n),\\
    \mathbf r_{n}&= \frac{1}{n}(\lambda_{-1}\mathbf f_n).
  \end{aligned}
\end{equation}

\begin{remark}
  For fixed $\mathbf F$ and $\beta$ the mapping
  \begin{equation*}
(\lambda_{-n-1}, \ldots \lambda_{-1}) \to (\mathbf r_{-n}, \ldots, \mathbf r_n)
  \end{equation*}
  is an injection from $\mathcal L^\beta$ into $\mathcal N^\beta \subset \mathcal R^\beta$.
\end{remark}

There is a good reason to choose the coefficients $\lambda_{-1}$, $\lambda_{-2}$, $\lambda_{-3}$ as dependent variables, provided that ${\mathbf f}_{0}$, ${\mathbf f}_1$, ${\mathbf f}_{2}$ are linearly independent. This way the coefficients $\lambda_{-n-1}, \ldots,\lambda_{-5}$ can be set to $0$ ensuring that the lowest curve coefficients $\mathbf r_{-n}, \ldots,\mathbf r_{-4}$ vanish. The coefficient $\lambda_{-4}$ must be nonzero because otherwise the solution would be trivial. The three coefficients $\lambda_{-3}$, $\lambda_{-2}$, $\lambda_{-1}$ are computed linearly from \eqref{eq:8}. This leads to a rational PH curve with minimal multiplicity of $\beta$ as zero of the denominator and is the generic case of \cite[Theorem~4.6]{kalkan22}.

If the coefficients ${\mathbf f}_{0}$, ${\mathbf f}_1$, ${\mathbf f}_{2}$ are linearly dependent (non-generic case) some other three coefficients must be chosen as dependent. There is a natural strategy how to generally obtain a basis of~$\mathcal N^{\beta}$:

\begin{remark}[Canonical basis of $\mathcal N^{\beta}$]
  \label{rem:canonical-basis}
Let $(i_1, i_2, i_3)$ be a triplet of indices so that ${\mathbf f}_{i_1}$, ${\mathbf f}_{i_2}$, ${\mathbf f}_{i_3}$ are linearly independent. Denote as
  \begin{equation*}
    D= \{\lambda_{(-1-i_1)},\lambda_{(-1-i_2)},\lambda_{(-1-i_3)}\}
  \end{equation*}
  the set of three dependent variables and
  \begin{equation*}
    I=\{\lambda_{-1}, \ldots, \lambda_{-n-1}\}\setminus D
  \end{equation*}
  the set of $n-2$ free variables. The basis curves are obtained by setting all variables of $I$ to $0$ except for one which is set to $1$, compute the variables of $D$ from \eqref{eq:7} and substitute the values into \eqref{eq:8}. In order to make this construction canonic, an a-priori rule for selecting the triplet $(i_1,i_2,i_3)$ is needed. It would be possible to pick the minimal or maximal indices with respect to a lexicographic ordering. In our examples (and in particular in Examples~\ref{ex:1}-\ref{ex:4} below) we usually have $n = \deg \mathbf F(t) = 4$ and select $(i_1,i_2,i_3) = (1,2,3)$ as this choice conforms with the constructions of~\cite{schroecker22}.
\end{remark}

We will demonstrate various cases on a series of examples.

\begin{example}
  \label{ex:1}
  Consider $\beta=0$ and
  \begin{equation*}
    \mathbf F(t)=
    \begin{pmatrix}
      1 \\
      1 \\
      1
    \end{pmatrix}
    +
    \begin{pmatrix}
      1 \\
      0 \\
      -1
    \end{pmatrix}
    t+
    \begin{pmatrix}
      1 \\
      1 \\
      0
    \end{pmatrix}
    t^2+
    \begin{pmatrix}
      0 \\
      1 \\
      -1
    \end{pmatrix}
    t^3+
    \begin{pmatrix}
      1 \\
      -1 \\
      1
    \end{pmatrix}
    t^4,
  \end{equation*}
  for which any three of the coefficients $\mathbf f_0$, $\mathbf f_1$, $\mathbf f_2$, $\mathbf f_3$, $\mathbf f_4$ are linearly independent. The space $\mathcal N^\beta$ has dimension $\deg \mathbf F(t) - 2 = 2$. A possible basis is $(\mathbf r_1(t), \mathbf r_2(t))$ with
  \begin{equation*}
    \small
    \begin{aligned}
      \mathbf r_1(t)&=
                      \begin{pmatrix}
                        -3 \\
                        -3 \\
                        -3
                      \end{pmatrix}
                      t^{-4}+
                      \begin{pmatrix}
                        -10 \\
                        -6 \\
                        -2
                      \end{pmatrix}
                      t^{-3}+
                      \begin{pmatrix}
                        -12 \\
                        -3 \\
                        12
                      \end{pmatrix}
                      t^{-2}+
                      \begin{pmatrix}
                        -6 \\
                        -24 \\
                        12
                      \end{pmatrix}
                      t^{-1}+
                      \begin{pmatrix}
                        12 \\
                        -30 \\
                        24
                      \end{pmatrix}
                      t+
                      \begin{pmatrix}
                        -3 \\
                        0 \\
                        0
                      \end{pmatrix}
                      t^2+
                      \begin{pmatrix}
                        -2 \\
                        2 \\
                        -2
                      \end{pmatrix}
                      t^3,\\
      \mathbf r_2(t)&=
                      \begin{pmatrix}
                        -2 \\
                        -2 \\
                        -2
                      \end{pmatrix}
                      t^{-3}+
                      \begin{pmatrix}
                        6 \\
                        9 \\
                        12
                      \end{pmatrix}
                      t^{-2}+
                      \begin{pmatrix}
                        6 \\
                        -12 \\
                        -24
                      \end{pmatrix}
                      t^{-1}+
                      \begin{pmatrix}
                        24 \\
                        -18 \\
                        12
                      \end{pmatrix}
                      t+
                      \begin{pmatrix}
                        -3 \\
                        18 \\
                        -12
                      \end{pmatrix}
                      t^2
                      +
                      \begin{pmatrix}
                        2 \\
                        2 \\
                        -2
                      \end{pmatrix}
                      t^3+
                      \begin{pmatrix}
                        3 \\
                        -3 \\
                        3
                      \end{pmatrix}
                      t^4.
    \end{aligned}
  \end{equation*}
  It is obtained from \eqref{eq:3} for
  \begin{equation*}
    \lambda_1(t) = 12t^{-5} + 18t^{-4}-6t^{-3}-6t^{-2}
    \quad\text{and}\quad
    \lambda_2(t) = 6t^{-4}-18t^{-3}+6t^{-2}+12t^{-1},
  \end{equation*}
  respectively.

  The pattern of non-vanishing Laurent coefficients of $\mathbf r_1(t)$ and $\mathbf r_2(t)$ at $t = \beta$ in above example is displayed in Figure~\ref{figEx} (top right) together with the patterns for the corresponding rational functions $\lambda_1(t)$ and $\lambda_2(t)$ (left) with filled dots. Using non-filled dots, we also display patterns for regular rational solutions and for polynomial solutions. Constant solutions are indicated with crossed dots. Obviously, the indicated PH curves are independent. In \cite{schroecker22} it has been shown that they also generate the spaces of regular and of polynomial solutions, respectively.
\end{example}

\begin{figure}
  \centering
  \includegraphics[]{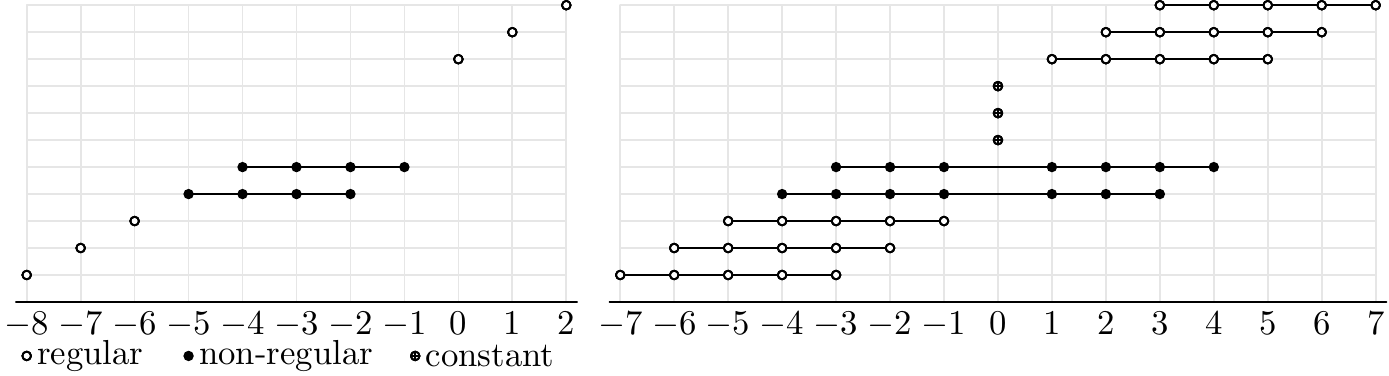}
\par\bigskip
  \includegraphics{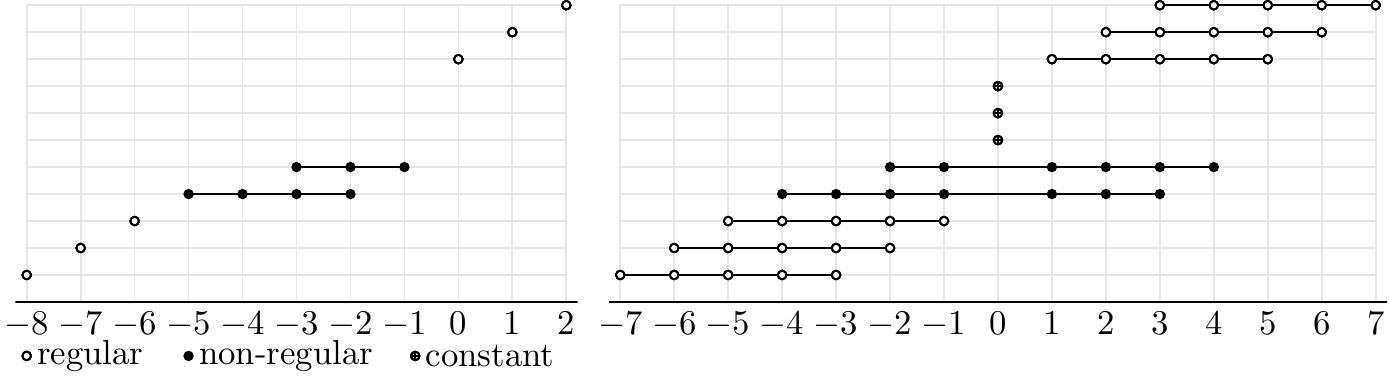}
\par\bigskip
  \includegraphics{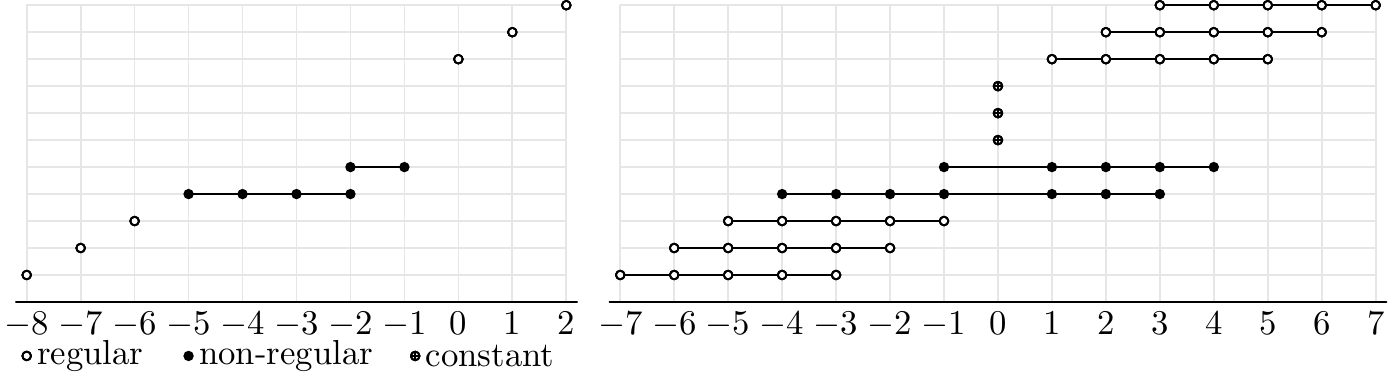}
\par\bigskip
  \includegraphics[]{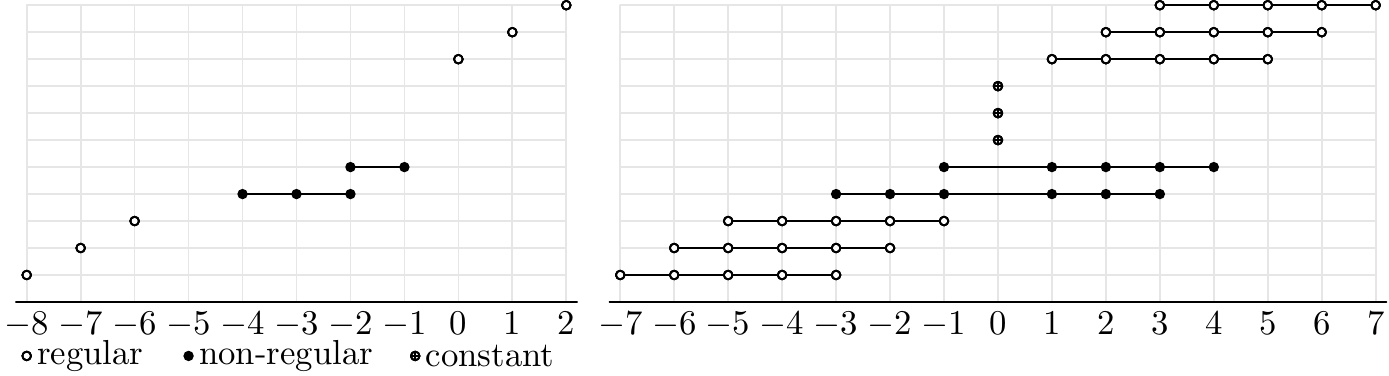}
\caption{Patterns of non-zero coefficients for the bases of $\mathcal L^\beta$ (left) and $\mathcal R^\beta$ (right) for Examples~\ref{ex:1}--\ref{ex:4} (from top to bottom). Note the three additional constant elements on the right.}
  \label{figEx}
\end{figure}

Note that the patterns of non-vanishing Laurent coefficients displayed in Figure~\ref{figEx}, top, are not particular for Example~\ref{ex:1} alone but appear for $\mathbf F(t)$ of degree four and $\beta$ such that any three of the coefficients of \eqref{eq:2} are independent. We will make use of this observation in Section~\ref{sec:examples}.

\begin{example}
  \label{ex:2}
  Consider $\beta=0$ and
  \begin{equation*}
    \mathbf F(t)=
    \begin{pmatrix}
      1 \\
      1 \\
      1
    \end{pmatrix}
    +
    \begin{pmatrix}
      1 \\
      0 \\
      0
    \end{pmatrix}
    t+
    \begin{pmatrix}
      0 \\
      1 \\
      1
    \end{pmatrix}
    t^2+
    \begin{pmatrix}
      1 \\
      1 \\
      0
    \end{pmatrix}
    t^3+
    \begin{pmatrix}
      1 \\
      -1 \\
      1
    \end{pmatrix}
    t^4
  \end{equation*}
  exhibiting the linear dependency of the vectors $\{\mathbf f_0, \mathbf f_1, \mathbf f_2\}$. The space $\mathcal N^\beta$ of non-regular has dimension $2$; a basis $(\mathbf r_1(t), \mathbf r_2(t))$ is given by
  \begin{equation*}
    \small
    \begin{aligned}
      \mathbf r_1(t)&=
                      \begin{pmatrix}
                        -3 \\
                        -3 \\
                        -3
                      \end{pmatrix}
                      t^{-4}+
                      \begin{pmatrix}
                        -12 \\
                        -8 \\
                        -8
                      \end{pmatrix}
                      t^{-3}+
                      \begin{pmatrix}
                        -6 \\
                        0 \\
                        0
                      \end{pmatrix}
                      t^{-2}+
                      \begin{pmatrix}
                        36 \\
                        0 \\
                        12
                      \end{pmatrix}
                      t^{-1}+
                      \begin{pmatrix}
                        12 \\
                        -72 \\
                        -12
                      \end{pmatrix}
                      t+
                      \begin{pmatrix}
                        -24 \\
                        -12 \\
                        -6
                      \end{pmatrix}
                      t^2+
                      \begin{pmatrix}
                        -12 \\
                        12 \\
                        -12
                      \end{pmatrix}
                      t^3,\\
      \mathbf r_2(t)&=
                      \begin{pmatrix}
                        6 \\
                        6 \\
                        6
                      \end{pmatrix}
                      t^{-2}+
                      \begin{pmatrix}
                        24 \\
                        12 \\
                        12
                      \end{pmatrix}
                      t^{-1}+
                      \begin{pmatrix}
                        0 \\
                        -24 \\
                        -12
                      \end{pmatrix}
                      t+
                      \begin{pmatrix}
                        -12 \\
                        6 \\
                        0
                      \end{pmatrix}
                      t^2+
                      \begin{pmatrix}
                        0 \\
                        8 \\
                        -4
                      \end{pmatrix}
                      t^3+
                      \begin{pmatrix}
                        3 \\
                        -3 \\
                        3
                      \end{pmatrix}
                      t^4.
    \end{aligned}
  \end{equation*}
  These basis vectors correspond to $\lambda_1(t) = 12t^{-5} + 24t^{-4} - 12t^{-3} -36t^{-2}$ and $\lambda_2(t) = -12t^{-3} - 12t^{-2} + 12t^{-1}$, respectively. The structure of the space $\mathcal {R^{\beta}}$ in terms of non-vanishing Laurent series coefficients of $\lambda(t)$ (left) and $\mathbf r(t)$ (right) is displayed at Figure~\ref{figEx} (second from top).
\end{example}

\begin{example}
  \label{ex:3}
  Consider $\beta=0$ and
  \begin{equation*}
    \mathbf F(t)=
    \begin{pmatrix}
      1 \\
      1 \\
      1 \\
    \end{pmatrix}
    +
    \begin{pmatrix}
      1 \\
      1 \\
      1
    \end{pmatrix}
    t+
    \begin{pmatrix}
      0 \\
      1 \\
      1
    \end{pmatrix}
    t^2+
    \begin{pmatrix}
      1 \\
      1 \\
      0
    \end{pmatrix}
    t^3+
    \begin{pmatrix}
      1 \\
      -1 \\
      1
    \end{pmatrix}
    t^4
  \end{equation*}
  exhibiting the linear dependency of the vectors $\{\mathbf f_0, \mathbf f_1\}$. The space $\mathcal N^\beta$ has dimension $2$. A basis $(\mathbf r_1(t), \mathbf r_2(t))$ is given by
  \begin{equation*}
    \small
    \begin{aligned}
      \mathbf r_1(t)&=
                      \begin{pmatrix}
                        -1 \\
                        _  -1 \\
                        -1
                      \end{pmatrix}
                      t^{-4}+
                      \begin{pmatrix}
                        -4 \\
                        -4 \\
                        -4
                      \end{pmatrix}
                      t^{-3}+
                      \begin{pmatrix}
                        -8 \\
                        -10 \\
                        -10
                      \end{pmatrix}
                      t^{-2}+
                      \begin{pmatrix}
                        0 \\
                        -8 \\
                        -4
                      \end{pmatrix}
                      t^{-1}+
                      \begin{pmatrix}
                        16 \\
                        -12 \\
                        -4
                      \end{pmatrix}
                      t+
                      \begin{pmatrix}
                        -2 \\
                        -10 \\
                        4
                      \end{pmatrix}
                      t^2+
                      \begin{pmatrix}
                        -4 \\
                        4 \\
                        -4
                      \end{pmatrix}
                      t^3, \\
      \mathbf r_2(t)&=
                      \begin{pmatrix}
                        12 \\
                        12 \\
                        12
                      \end{pmatrix}
                      t^{-1}+
                      \begin{pmatrix}
                        12 \\
                        0 \\
                        0
                      \end{pmatrix}
                      t+
                      \begin{pmatrix}
                        -6 \\
                        0 \\
                        6
                      \end{pmatrix}
                      t^2+
                      \begin{pmatrix}
                        0 \\
                        8 \\
                        -4
                      \end{pmatrix}
                      t^3+
                      \begin{pmatrix}
                        3 \\
                        -3 \\
                        3
                      \end{pmatrix}
                      t^4.
    \end{aligned}
  \end{equation*}
  It corresponds to $\lambda_1(t) = 4t^{-5}+8t^{-4}+8t^{-3}-12t^{-2}$ and $\lambda_2(t) = -12t^{-2}+12t^{-1}$, respectively. The structure of the space $\mathcal {R^{\beta}}$ is displayed at Figure~\ref{figEx}, second from bottom.
\end{example}

\begin{example}
  \label{ex:4}
  Consider $\beta=0$ and
  \begin{equation*}
    \mathbf F(t)=
    \begin{pmatrix}
      1 \\
      1 \\
      1
    \end{pmatrix}
    +
    \begin{pmatrix}
      1 \\
      1 \\
      1
    \end{pmatrix}
    t+
    \begin{pmatrix}
      1 \\
      0 \\
      0
    \end{pmatrix}
    t^2+
    \begin{pmatrix}
      0 \\
      1 \\
      1
    \end{pmatrix}
    t^3+
    \begin{pmatrix}
      1 \\
      -1 \\
      1
    \end{pmatrix}
    t^4
  \end{equation*}
  exhibiting the linear dependency of the vectors $\{\mathbf f_0, \mathbf f_1\}$ and also of the vectors $\{\mathbf f_0, \mathbf f_2, \mathbf f_3\}$. Then the space $\mathcal N^\beta$ has dimension $2$. A basis $(\mathbf r_1(t), \mathbf r_2(t))$ is given by
  \begin{equation*}
    \begin{aligned}
      \mathbf r_1(t)&=
                      \begin{pmatrix}
                        -2 \\
                        -2 \\
                        -2
                      \end{pmatrix}
                      t^{-3}+
                      \begin{pmatrix}
                        -6 \\
                        -6 \\
                        -6
                      \end{pmatrix}
                      t^{-2}+
                      \begin{pmatrix}
                        -6 \\
                        0 \\
                        0
                      \end{pmatrix}
                      t^{-1}+
                      \begin{pmatrix}
                        0 \\
                        0 \\
                        12
                      \end{pmatrix}
                      t+
                      \begin{pmatrix}
                        3 \\
                        -6 \\
                        0
                      \end{pmatrix}
                      t^2+
                      \begin{pmatrix}
                        -2 \\
                        2 \\
                        -2
                      \end{pmatrix}
                      t^3,\\
      \mathbf r_2(t)&=
                      \begin{pmatrix}
                        -12 \\
                        -12 \\
                        -12
                      \end{pmatrix}
                      t^{-1}+
                      \begin{pmatrix}
                        0 \\
                        -12 \\
                        -12
                      \end{pmatrix}
                      t+
                      \begin{pmatrix}
                        -6 \\
                        6 \\
                        6
                      \end{pmatrix}
                      t^2+
                      \begin{pmatrix}
                        4 \\
                        -8 \\
                        0
                      \end{pmatrix}
                      t^3+
                      \begin{pmatrix}
                        -3 \\
                        3 \\
                        -3
                      \end{pmatrix}
                      t^4.
    \end{aligned}
  \end{equation*}
  It corresponds to $\lambda_1(t) = -6t^{-4}+6t^{-3}-6t^{-2}$ and $\lambda_2(t) = 12t^{-2}-12t^{-1}$, respectively. The structure of the space $\mathcal {R^{\beta}}$ is displayed at Figure~\ref{figEx}, bottom. Once more, it is typical for all spaces $\mathcal R^\beta$ of rational PH curves with $\deg \mathbf F(t) = 4$ where $\{\mathbf f_0,\mathbf f_1\}$ as well as $\{\mathbf f_0, \mathbf f_2, \mathbf f_3\}$ are linearly dependent. It is noteworthy that this rather special case already appeared in literature \cite{krajnc3}. We will encounter it again in Section~\ref{sec:examples}.

\end{example}

\section{An Optimal Interpolation Scheme}
\label{sec:interpolation}

We are now going to present an optimal interpolation scheme for the rational curves constructed in the preceding sections. In doing so, we focus on \emph{PH curves} as they have received a lot of attention in literature. Rational PH curves can be defined as rational curves with a rational tangent indicatrix. They are subsumed into our approach for the special choice $\mathbf F(t) = \mathcal A(t) \qi \Cj{\mathcal A}(t)$ of the vector valued polynomial $\mathbf F(t) \in \mathbb R^3[t]$, see \cite{kalkan22} for a more detailed explanation. Here, $\mathcal A(t) \in \mathbb H[t]$ is a polynomial with quaternion coefficients and $\qi$ is a quaternion unit. Note that $\mathcal A(t) \qi \Cj{\mathcal A}(t)$ is vectorial so that it can indeed by identified with a vector valued polynomial in~$\mathbb R^3[t]$.

Our optimal interpolation procedure requires the construction of a vector space $\mathcal Q$ of rational PH curves (``interpolation space'') from which an optimal interpolant is computed. The construction of $\mathcal Q$ and also the optimization routine is based on our zero-residuum approach developed in Section~\ref{sec:residuum} that, with above choice of $\mathbf F(t)$, computes rational PH curves by direct integration. In particular, we select the interpolation space $\mathcal Q$ as subspace of the vector space $\mathcal R$ of rational PH curves whose tangent directions are given by the fixed polynomial $\mathbf F(t) = \mathcal A(t) \qi \Cj{\mathcal A}(t)$. If $(\mathbf q_1(t), \mathbf q_2(t), \ldots, \mathbf q_d(t))$ is a basis of $\mathcal Q$, we may write
\begin{equation}
  \label{eq:9}
  \mathbf r(t) = \sum_{i=1}^d \varrho_i \mathbf q_i(t)
\end{equation}
with coefficients $\varrho_1$, $\varrho_2$, \ldots, $\varrho_d \in \mathbb R$ that are determined by the following:
\begin{description}
\item[Interpolation constraints:] Hermite interpolation constraints in start point $\mathbf r(0)$ and end-point $\mathbf r(1)$, respectively, provide linear constraints on $\varrho_1$, $\varrho_2$, \ldots, $\varrho_d$. By design, the selection of $\mathbf F(t)$, interpolation of derivative data is subject to some limitations. We discuss this in more detail in Section~\ref{sec:interpolation-constraints}
\item[Quadratic functionals:] To the linear interpolation constraints we add one (at most) quadratic functional to be satisfied in optimal sense. The standard functional to be minimized is the integral of the squared norm of the derivative vector but -- due to the PH property of all curves in $\mathcal Q$ -- also minimal or approximate arc-length is possible, cf. Section~\ref{sec:quadratic-optimization}.
\item[Shape properties:] A common and annoying phenomenon in interpolation with curves in dual representation is a tendency for cusps to appear, see e.g. \cite[Example 2]{FaroukiSir2}. With our approach it is possible to avoid this by adding linear inequalities that ensure the absence of zeros of $\mathbf r'(t)$ in the interpolation interval. A crucial role in this context play the basis vector's respective speed functions $\lambda_1(t)$, $\lambda_2(t)$, \ldots $\lambda_d(t)$, cf. Section~\ref{sec:cusp-avoidance}.
\end{description}

Combining interpolation, optimization of quadratic functional, and shape properties in above sense produces a quadratic program that can be solved efficiently by standard procedures.

Let us stress an important advantage of our approach. Quadratic optimization will typically result in an  interpolant $\mathbf r(t)$ in \emph{numeric form.} Since the PH property is inherently of symbolic nature, the numeric (and therefore necessarily approximate) representation might destroy this property and the ability to obtain the rational speed function. But in our case we can still recover the speed function as $\Vert \mathbf r'(t) \Vert = \pm \lambda(t) \mathcal{A}(t) \Cj{\mathcal A}(t)$. Also notice that the absence of zeros of $\lambda(t)$ allows to resolve the sign ambiguity on the interpolation interval.

\subsection{Interpolation}
\label{sec:interpolation-constraints}

Various case of the Hermite interpolation were studied in the context of polynomial Pythagorean hodograph curves, cf. \cite{Juettler:99b,FaroukiSpace2,PHhelices,sir05,SirC2,C2length}. We are concerned with Hermite interpolation in either analytic or geometric sense. In any case, the rational PH curve \eqref{eq:9} is supposed to satisfy $\mathbf r(0) = \mathbf p_0$ and $\mathbf r(1) = \mathbf p_1$ where $\mathbf p_0$, $\mathbf p_1 \in \mathbb R^3$ are given points. Each of these interpolation conditions imposes three independent constraints on the coefficients $\varrho_0$, $\varrho_1$, \ldots, $\varrho_d$.

It is not possible to enforce the interpolation of arbitrary derivative vectors $\mathbf v_0$, $\mathbf v_1$ or tangent directions by a proper choice of the coefficients $\varrho_i$ as the tangent directions are already determined by the choice of $\mathbf F(t)=\mathcal A(t) \qi \Cj{\mathcal A}(t)$. For this reason we first construct $\mathcal A(t)$ so that the tangent directions (and possibly also the osculating planes) are matched. This can be done solving quadratic/linear equations in the domain of quaternions, see e.g. \cite{SirC2}. Comparing to the usual interpolation with polynomial PH curves the requirements put on $\mathcal A(t)$ are relaxed, because the additional degrees of freedom $\rho_i$ can be used in the following way.

Because of
\begin{equation*}
  \mathbf r'(t) = \lambda(t) \mathbf F(t),
\end{equation*}
adjusting the lengths of the derivative vector ($C^1$ interpolation condition) imposes just one linear constraint for start and end velocity, respectively.

The second derivative vector of $\mathbf r(t)$ equals
\begin{equation*}
  \mathbf r''(t) = \lambda'(t) \mathbf F(t) + \lambda(t) \mathbf F'(t).
\end{equation*}
Therefore, exact interpolation of acceleration vectors $\mathbf r''(0) = \mathbf a_0$ and $\mathbf r''(1) = \mathbf a_1$ (in addition to exact interpolation of velocity vectors) is only possible if $\vec{a}_0$ lies on the line through $\lambda(0)\mathbf F'(0)$ and parallel to $\mathbf F(0)$ and similar for $\vec{a}_1$. In this case, acceleration interpolation imposes one linear constraint each.

Combining these observations with the well-known formula
\begin{equation*}
  \varkappa(t) = \frac{\Vert \mathbf r'(t) \times \mathbf r''(t) \Vert}{\Vert \mathbf r' \Vert^3}
\end{equation*}
for the curvature of a parametric space curve \cite[Equation~(3.7)]{banchoff10} we see that $\mathbf r'(0)$ together with $\mathbf F(0)$ and $\mathbf F'(0)$ fully determines the curvature $\varkappa(0)$. In other words, $C^1$ interpolation determines curvature as soon as $\mathbf F(t)$ is fixed.

\subsection{Quadratic Functionals}
\label{sec:quadratic-optimization}

The most common quadratic functionals to be minimized in curve interpolation or approximation problems is likely the curve energy
\begin{equation}
  \label{eq:10}
  E(\mathbf r) = \int_0^1 \langle \mathbf r'(t), \mathbf r'(t) \rangle \;\mathrm{d}t.
\end{equation}
It is quadratic in the interpolation variables $\varrho_0$, $\varrho_1$, \ldots, $\varrho_d$. An interesting alternative in certain applications might be the arc-length
\begin{equation}
  \label{eq:11}
  L(\mathbf r) = \int_0^1 \sqrt{\langle \mathbf r'(t), \mathbf r'(t) \rangle} \;\mathrm{d}t
\end{equation}
which, by the PH property, is piecewise linear in the interpolation variables. Because we are optimizing over a cusp-free range, $\mathbf r'(t)$ does not change sign on the interval $[0,1]$, so that sign ambiguity that is implicit in the square root of \eqref{eq:11} can be resolved and the optimization target functional is even \emph{linear.}

There are some recent attempts of interpolation with (polynomial) PH curves of prescribed length $s > 0$ \cite{farouki16,knez22,zagar23}. Assuming that exact and feasible solutions exist in the interpolation space $\mathcal Q$, they will be found by minimizing the quadratic functional
\begin{equation}
  \label{eq:12}
  L_s(\mathbf r) = (L(\mathbf r) - s)^2
\end{equation}
where $L$ is taken from~\eqref{eq:11}.

We refer the reader to Section~\ref{sec:examples} for examples where each of linear/quadratic functionals \eqref{eq:10}--\eqref{eq:12} is minimized.

\subsection{Cusp Avoidance}
\label{sec:cusp-avoidance}

Any basis vector $\mathbf q_i(t)$ of $\mathcal Q$ satisfies $\mathbf q'_i(t) = \lambda_i(t) \mathbf F(t)$ for the rational speed function $\lambda_i(t)$. By linearity of the derivative, \eqref{eq:9} yields
\begin{equation*}
  \mathbf r'(t) = \lambda(t) \mathbf F(t)
  \quad\text{where}\quad
  \lambda(t) = \sum_{i=1}^d \varrho_i \lambda_i(t).
\end{equation*}
Each of the speed functions $\lambda_i(t)$ can be written in reduced form as $\lambda_i(t) = \nu_i(t)/\delta_i(t)$ with polynomials $\nu_i(t)$ and $\delta_i(t)$. Set $\delta(t) \coloneqq \lcm(\delta_1(t), \delta_2(t), \ldots, \delta_d(t))$ and denote by $\mu_i(t)$ the polynomial implicitly defined by $\lambda_i(t) = \mu_i(t)/\delta(t)$. Then the numerator of $\lambda(t)$ can be computed as
\begin{equation}
  \label{eq:13}
  \mu(t) = \sum_{i=1}^d \varrho_i\mu_i(t).
\end{equation}
In order to avoid cusps, we need to avoid zeros of $\mathbf r'(t)$. These are precisely the zeros of $\mu(t)$. Hence the relevance of this polynomial. For lack of a better name, we will simply refer to it as the ``$\mu$-polynomial of $\mathbf r(t)$''.

Since the interpolation variables $\varrho_i$ appear linearly in \eqref{eq:13}, we do not destroy the quadratic character of our optimization problem by imposing linear or quadratic constraints on the coefficients of $\mu(t)$. A straightforward way to avoid zeros of $\mu(t)$ on the open interval $(0,1)$ is to write $\mu(t) = \sum_{j=0}^m B^m_j(t) m_j$ where $B^m_j(t)$ is the $j$-th Bernstein polynomial of degree $m$ and require all Bernstein coefficients $m_j$ to be either non-negative or non-positive, whatever is suggested by a suitable initial solution. This provides linear constraints on the interpolation variables $\varrho_0$, $\varrho_1$, \ldots, $\varrho_d$. We emphasize that these constitute a sufficient but not a necessary condition for non-negativity of $\mu(t)$ on $(0,1)$. It is possible to loosen the constraints a bit by formally elevating the degree of $\mu(t)$. However, only a minor effect on optimal solutions is to be expected. This is confirmed in our experiments in the subsequent Section~\ref{sec:examples}.

\section{Examples}
\label{sec:examples}

We have so far provided a detailed but rather general description of an optimal interpolation procedure with rational PH curves. Some concrete questions, in particular how to pick a suitable vector space $\mathcal Q$ and a basis $(\mathbf q_1(t), \mathbf q_2(t), \ldots, \mathbf q_d(t))$, shall be answered now at hand of examples. In line with our statement in Section~\ref{sec:interpolation-constraints}, two of them are taken from literature.

All computations in this section were carried out on a standard desktop PC using
a recent version of Maple\texttrademark \cite{maple} and its
\texttt{Optimiziation[QPSolve]} command for solving quadratic programs.

\subsection*{$G^1$ and $C^1$ Interpolation by Slant Helices}

The canonical way of solving $G^1$ or $C^1$ interpolation problems using PH curves \cite{Juettler:99b,FaroukiSpace2,sir05} requires the solution of an algebraic system of equations. These approaches are well-understood and straightforward, and they produce interpolants of minimal degree. However, in the context of this paper, this minimal degree does not provide a significant advantage as it is lost during quadratic optimization (cf. the example in the next Section). If one is willing to accept sub-optimal curve degrees, it is possible to use our approach in a straightforward manner. This will be described at hand of an example.

The data to be interpolated in $G^1$ or $C^1$ sense is
\begin{equation*}
  \mathbf p_0 = \begin{pmatrix} 0 \\ 0 \\ 0 \end{pmatrix},\quad
  \mathbf p_1 = \frac{1}{6}\begin{pmatrix} 3 \\ 6 \\ 2 \end{pmatrix},\quad
  \mathbf v_0 = \begin{pmatrix} 1 \\ 0 \\ 0 \end{pmatrix},\quad
  \mathbf v_1 = \frac{1}{7} \begin{pmatrix} -1 \\ 2 \\ 2 \end{pmatrix}.
\end{equation*}

In order to solve the $G^1$ interpolation problem, we start by computing a quaternion polynomial $\mathcal A(t)$ that satisfies $\mathcal A(0) \times \mathbf p_0 = \mathcal A(1) \times \mathbf p_1 = 0$. Even if restricting ourselves to the case $\deg \mathcal A(t) = 1$ we obtain a one-parametric solution set from which we select
\begin{equation*}
  \mathcal A(t) = 1 - (1 - \qi - \qj - \qk)t
  \quad\text{whence}\quad
  \mathbf F(t) = 4\qk t^2 - 2(\qi - \qj + \qk) t + \qi.
\end{equation*}
The condition $\deg \mathcal A(t) = 1$ ensures that all curve tangents form a constant angle with some fixed direction. Curves with this property are referred to as ``curves of constant slope'' or ``slant helices'' \cite{barros}. Hence the name of this section.

We select an interpolation space $\mathcal Q$ that is spanned by the three constant solutions $\mathbf q_0 = (1, 0, 0)$, $\mathbf q_1 = (0, 1, 0)$, $\mathbf q_2 = (0, 0, 1)$, one rational solution
\begin{equation}
  \label{eq:14}
  \mathbf q_3(t) = \sum_{\substack{k=-3\\k \neq0}}^1 \mathbf q_{t,k}(t - \beta)^k =
  \begin{pmatrix} 3 \\ -2 \\ 6 \end{pmatrix} (t + 1)^{-3} +
  \begin{pmatrix} -3 \\ 3 \\ -15 \end{pmatrix} (t + 1)^{-2} +
  \begin{pmatrix} 0 \\ 0 \\ 12 \end{pmatrix} (t + 1)^{-1}
\end{equation}
to $\beta = -1$ and the polynomial solutions
\begin{equation}
  \label{eq:15}
  \mathbf q_{\ell+2}(t) = \int (n+\ell)t^{\ell-1} \mathbf F(t) \; \mathrm dt,
  \quad
  \ell \in \{2,3,4\}.
\end{equation}
The factor $(n + \ell)$ in \eqref{eq:15} is for normalization purposes only. It ensures that the respective leading coefficients of $\mathbf q_{\ell+2}(t)$ and $\mathbf F(t)$ are equal.

By construction, all basis vectors match start and end tangent of our interpolation problem. The polynomial basis vector $\mathbf q_4(t)$ satisfies $\mathbf q_4(0) = \mathbf p_0$ but $\mathbf q_4(1) \neq \mathbf p_1$. Nonetheless, we use it as to initialize the quadratic program. With this choice, the optimization routine with respect to the energy functional \eqref{eq:11} succeeds and produces the solution $\mathbf{r}_o(t)$ displayed in Figure~\ref{fig:G1-Simple-R1}.

\begin{figure}
  \centering
  \includegraphics[]{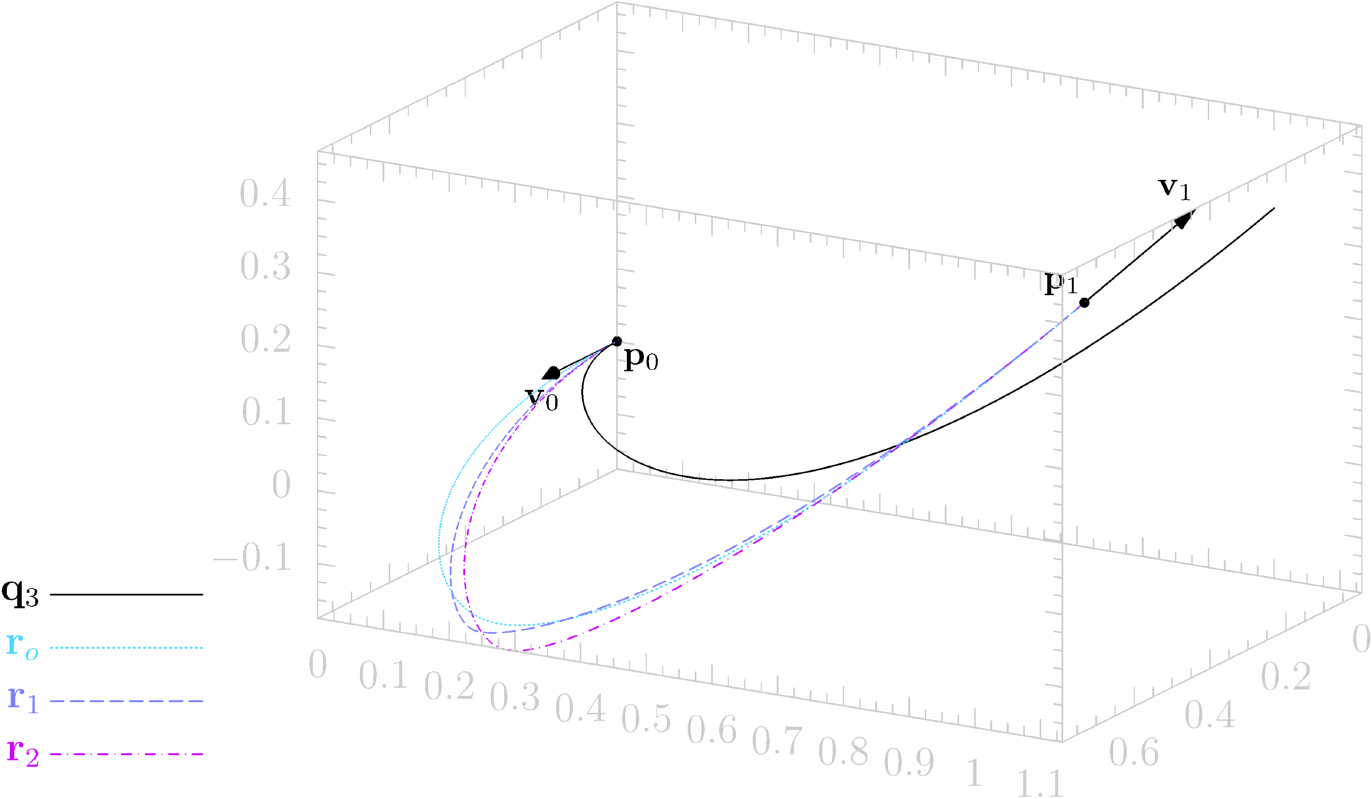}
  \caption{Optimal $G^1$ or $C^1$ interpolation by rational PH curves. The initial ``solution'' $\mathbf q_3(t)$ fails to meet the end-point interpolation condition. The arrows indicating the unit vectors $\mathbf v_0$ and $\mathbf v_1$ are scaled to $20$\% of their actual size.}
  \label{fig:G1-Simple-R1}
\end{figure}

In order to solve the $C^1$ interpolation problem, we enlarge the interpolation space $\mathcal Q$ by adding further polynomial or rational basis vectors. It is natural to use the optimal solution $\mathbf r_o(t)$ of the $G^1$ interpolation problem to initialize the quadratic optimization. With this choice, the optimization routine finds a feasible solution only after augmenting the basis of $\mathcal Q$ by either
\begin{itemize}
\item five further polynomial solutions of the shape \eqref{eq:15} but with $\ell \in \{5, 6, 7, 8\}$ or
\item two further polynomial solutions of the shape \eqref{eq:15} with $\ell \in \{5, 6\}$ and a rational solution similar to \eqref{eq:14} but with denominator root $\beta = -2$.
\end{itemize}
In both cases, it is necessary to elevate the degree of the polynomial $\mu(t)$. The respective optimal solutions $\mathbf r_1(t)$ and $\mathbf r_2(t)$ are displayed in Figure~\ref{fig:G1-Simple-R1} as well. The respective optimal values of the energy functional \eqref{eq:11} are comparable:
\begin{equation*}
  E(\mathbf r_o) = 5.7766,\quad
  E(\mathbf r_1) = 5.8426,\quad
  E(\mathbf r_2) = 6.1102.
\end{equation*}
As expected, the fewer restrictions of $G^1$ interpolation allows for a slightly better value.

The interpolation approach in this example is very straightforward and easy to use. However, a clear disadvantage is the necessity to come up with a feasible initial point. This can be circumvented if one only wishes to improve upon an already existing solution. For the examples to follow, we will take this point of view.

\subsection*{Rational $C^1$ Interpolation}

A general solution to the $C^1$ interpolation problem by quintic polynomial PH curves was presented in \cite{FaroukiSpace2}. The authors of \cite{sir05} suggest to pick a solution from a two-dimensional variety of solution curve that is optimal in the sense of approximation order. We consider the interpolation data
\begin{equation}
  \label{eq:16}
  \mathbf p_0 = \begin{pmatrix}0 \\ 0 \\ 0\end{pmatrix},\quad
  \mathbf p_1 = \frac{1}{11520}\begin{pmatrix}34207 \\ -12208 \\ 22848\end{pmatrix},\quad
  \mathbf v_0 = \frac{1}{2}\begin{pmatrix}12 \\ 5 \\ 0\end{pmatrix},\quad
  \mathbf v_1 = \frac{1}{57600}\begin{pmatrix}316151 \\ -144000 \\ 0\end{pmatrix}
\end{equation}
which results in the optimal quintic PH interpolant
\begin{equation}
  \label{eq:17}
  \mathbf r(t) = \frac{1}{57600}
  \begin{pmatrix}
    -28517t^5+113520t^4+178192t^3-437760t^2+345600t \\
    143472t^5-466704t^4+625072t^3-506880t^2+144000t \\
    -274176t^5+695232t^4-796416t^3+489600t^2
  \end{pmatrix}.
\end{equation}
The associated polynomial $\mathcal A(t)$ is quadratic and reads as
\begin{equation*}
  \mathcal A(t) = \frac{1}{240}
  \bigl(
    (840\qi + 427\qj - 816\qk) t^2  + (-864\qi - 672\qj + 816\qk) t + 600\qi + 120\qj
  \bigr).
\end{equation*}
The degree of the polynomial tangent vector field $\mathbf F(t) \coloneqq \mathcal{A}(t) \qi \Cj{\mathcal{A}(t)}$ equals $n = 4$. We optimize with respect to the curve energy \eqref{eq:10} which, for this example, equals $E(\mathbf r) \approx 17.75$.

\paragraph*{Polynomial PH Curves}

In a first attempt, we pick the interpolation space $\mathcal Q = \mathcal{P}^p$ of polynomial PH curves of degree at most $p \ge n+1 = 5$ to $\mathcal{A}(t)$. A basis for this vector space is given by the three constant vectors $\mathbf q_0 = (1, 0, 0)$, $\mathbf q_1 = (1, 0, 0)$, $\mathbf q_2 = (1, 0, 0)$ plus the polynomial PH curves
\begin{equation}
  \label{eq:18}
  \mathbf q_{\ell+2}(t) = \int (n + \ell)t^{\ell-1} \mathbf F(t) \;\mathrm{d}t,
  \quad \ell \in \{1, 2, \ldots, p - n\}.
\end{equation}
The constant basis vectors correspond to the $\mu$-polynomials $\mu_0(t) = \mu_1(t) = \mu_2(t) = 0$, while we have $\mu_{\ell+2}(t) = (\ell+n)t^{\ell-1}$ for $1 \le \ell \le p-n$. The ``no-cusp constraint'' on $\mu(t) = \sum_{\ell=1}^{p-n+3} \mu_\ell(t)$ thus simply boils down to the non-negativity of $\varrho_3$, $\varrho_4$, \ldots, $\varrho_{p-n+2}$. Since $C^1$ interpolation in our context imposes eight independent linear constraints, we expect solutions different from \eqref{eq:17} for $\dim \mathcal Q \ge 9$, that is $p \ge 6 + n = 10$.

\begin{table}
  \centering
  \begin{tabular}{c|cccccccc}
    \small
    $p$ & $5, \ldots, 9$ & $10$     & $11$     & $12$     & $13$     & $14$    & $\cdots$ & $21$ \\
    \hline
    $\dim \mathcal Q$ & $4,\ldots,8$ & $9$     & $10$     & $11$     & $12$     & $13$    & $\cdots$ & $19$ \\
    $\deg \mathbf r_p$ & $5$      & $10$    & $11$    & $12$    & $13$    & $14$ & $\cdots$    & $21$ \\
    $E(\mathbf r_p)$     & $17.75$  & $17.70$ & $17.62$ & $17.58$ & $17.57$ & $17.56$ & $\cdots$ & $17.55$
  \end{tabular}
  \caption{Dimension of interpolation space $\mathcal Q$ (polynomials only), degree of optimal solution $\mathbf r_p$, and curve energy $E(\mathbf r_p)$.}
  \label{tab:C1-Ex3-P}
\end{table}

Denote the optimal solution to a given value of $p$ by $\mathbf r_p(t)$. Table~\ref{tab:C1-Ex3-P} shows the results of our optimization procedure. As expected, a value of $p \le 9$ just re-produces the original curve $\mathbf r(t)$. Starting with $p = 10$, some improvement can be observed. The gain as measured by the energy functional $E$ is rather small in both, absolute and relative terms, and the data suggests that going beyond $p = 12$ brings hardly any improvement. Figure~\ref{fig:C1-Ex3-P} displays the original solution $\mathbf r(t)$ and the optimal solution $\mathbf r_{12}(t)$ of curve degree~$12$.

\begin{figure}
  \centering
  \includegraphics[]{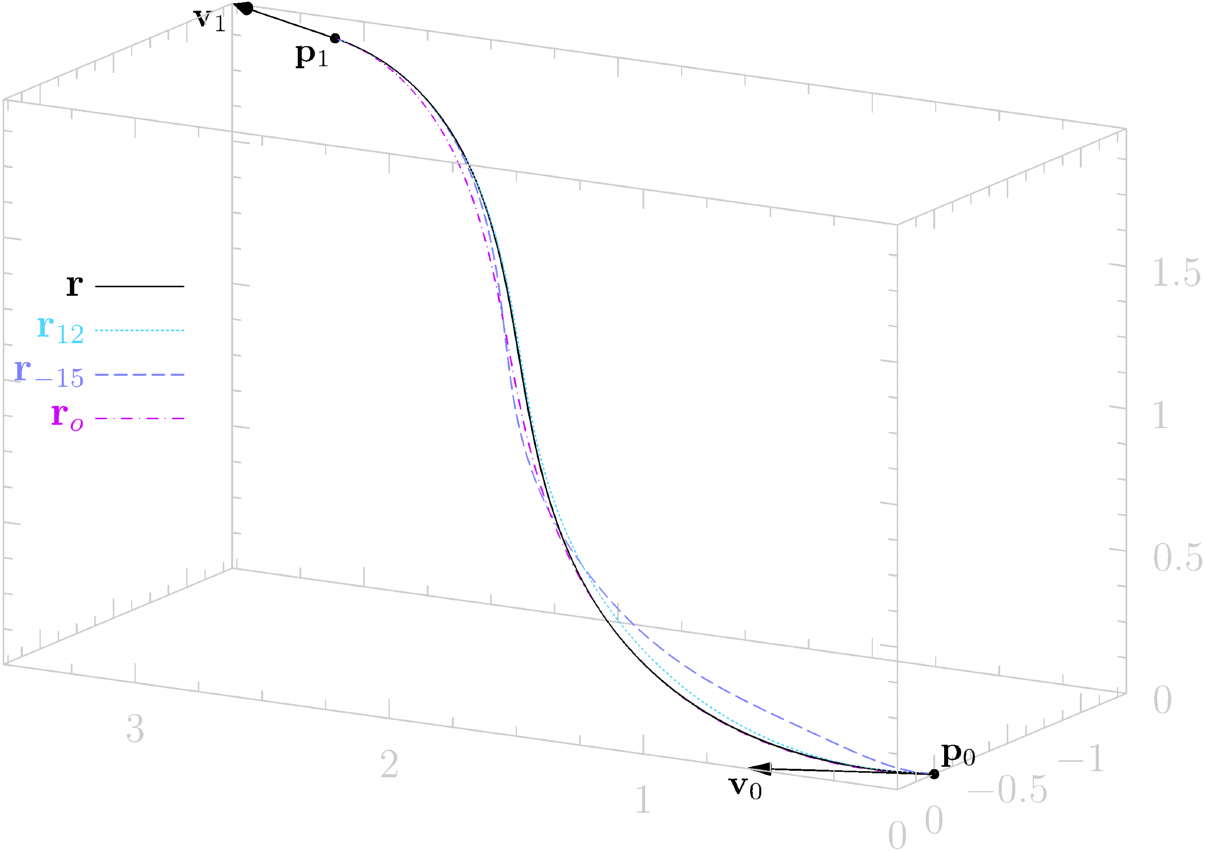}
  \caption{Optimized $C^1$ polynomial interpolant $\mathbf r_{12}(t)$, truly rational interpolant $\mathbf r_{-15}(t)$, and mixed polynomial-rational interpolant $\mathbf r_o(t)$. The original curve is $\mathbf r(t)$. The arrows indicating the unit vectors $\mathbf v_0$ and $\mathbf v_1$ are scaled to 10\% of their actual size.}
  \label{fig:C1-Ex3-P}
\end{figure}

\paragraph{Truly Rational PH Curves}

Now lets try a similar approach using mostly truly rational (non-polynomial) PH curves. We pick a value $\beta \notin [0,1]$, say $\beta = -1$, and, for $p \le -5$, we consider the interpolation space $\mathcal Q = \mathcal R^p$ with basis $(\mathbf q_1, \mathbf q_2, \ldots, \mathbf q_{2-p})$ where
\begin{itemize}
\item $\mathbf q_1 = (1, 0, 0)$, $\mathbf q_2 = (1, 0, 0)$, $\mathbf q_3 = (1, 0, 0)$ are the three constant basis vectors,
\item the two curves
  \begin{equation*}
    \mathbf q_4(t) = \sum_{\substack{k=-3\\k \neq0}}^4 \mathbf q_{4,k}(t-\beta)^k,
    \quad
    \mathbf q_5(t) = \sum_{\substack{k=-4\\k \neq0}}^3 \mathbf q_{5,k}(t-\beta)^k,
  \end{equation*}
  normalized by $\mathbf q_4(\beta) = \mathbf q_5(\beta) = \mathbf F(\beta)$, span the space of non-regular solutions (we are in a generic situation, similar to Example~\ref{ex:1}),
\item the curves
  \begin{equation*}
    \mathbf q_\ell(t) = \sum_{k=1-\ell}^{4-\ell} \mathbf q_{\ell,k}(t-\beta)^k,
    \quad
    \ell \in \{6, 7, \ldots, 1-p\}
  \end{equation*}
  span the space of regular solutions, and \item $\mathbf q_{2-p}(t)$ is the original curve $\mathbf r(t)$ given in~\eqref{eq:17}.
\end{itemize}
We add the polynomial solution $\mathbf q_{2-p}(t) = \mathbf r(t)$ to our interpolation space in order to ensure that our optimal solution really improves on it. It also is a feasible initial solution. The interpolation space is of dimension $\dim \mathcal Q = 2-p$. Its elements are PH curves with a Laurent expansion $\sum_{k=p}^5 \mathbf r_k(t-\beta)^k$ at $t = \beta$.

The $\mu$-polynomials corresponding to this basis are as follows:
\begin{itemize}
\item The constant basis vectors have vanishing $\mu$-polynomials, that is, $\mu_1(t) = \mu_2(t) = \mu_3(t) = 0$.
\item We have
  \begin{equation*}
    \lambda_4(t) = \sum_{k=-5}^{-2} \lambda_{4,k}(t-\beta)^k
    \quad\text{and}\quad
    \lambda_5(t) = \sum_{k=-4}^{-1} \lambda_{5,k}(t-\beta)^k
  \end{equation*}
  by Example~\ref{ex:1} (cf. also Figure~\ref{figEx}, top). Thus, $\mu_4(t) = \sum_{k=-5}^{-2}\lambda_{4,k}(t-\beta)^{k+p}$ and $\mu_5(t) = \sum_{k=-4}^{-1}\lambda_{5,k}(t-\beta)^{k+p}$.
\item We have $\lambda_\ell(t) = (t-\beta)^{-\ell}$ and therefore $\mu_\ell(t) =
  (t-\beta)^{p-\ell}$ for $\ell \in \{6,7, \ldots, 1-p\}$.
\item The $\mu$-polynomial to the polynomial solution $\mathbf q_{2-p}(t) = \mathbf r(t)$ of degree five equals $\mu_{2-p}(t) = 1$.
\end{itemize}

\begin{table}
  \centering
  \begin{tabular}{c|cccccccc}
    \small
    $p$ & $-3, \ldots, -10$ & $-11$ & $-12$ & $-13$ & $-14$ & $-15$ & $\cdots$ & $-20$ \\
    \hline
    $\dim \mathcal Q$ & $5, \ldots, 12$ & $13$ & $14$ & $15$ & $16$ & $17$ & $\cdots$ & $22$ \\
    $\deg \mathbf r_p$ & $5$ & $16$ & $17$ & $18$ & $19$ & $20$ & $\cdots$ & $25$ \\
    $E(\mathbf r_p)$ & $17.75$ & $17.75$ & $17.61$ & $17.64$ & $17.50$ & $17.44$ & $\cdots$ & $15.12$
  \end{tabular}
  \caption{Dimension of interpolation space $\mathcal Q$ (truly rational curves
    mostly), degree of optimal solution $\mathbf r_p$, and curve energy
    $E(\mathbf r_p)$.}
  \label{tab:C1-Ex3-R1}
\end{table}

Denote the optimal solution with respect to the curve energy \eqref{eq:10} by $\mathbf r_p(t)$. Some results for our optimization for decreasing $p$ (increasing dimension of $\mathcal Q$) are displayed in Table~\ref{tab:C1-Ex3-R1}. We can observe the following:
\begin{itemize}
\item Only for $\dim \mathcal Q = 13$ we find better solution in terms of the curve energy $E(\mathbf r_p)$. The improvement is very small.
\item There is a slight increase in curve energy when passing from dimension $14$ to dimension $15$. This can be explained by the non-convexity of the quadratic program so that only a local minimum is found.
\item For increasing dimension of $\mathcal Q$ the minimum drops way beyond the minimal value of $E(\mathbf r)$ that can be achieved with polynomial solutions only. The corresponding solution curves are, however, not visually appealing. This is illustrated in Figure~\ref{fig:C1-Ex3-P} where the optimal solution $\mathbf r_{-15}(t)$ of degree $20$ is depicted. Its energy is approximately $17.44$.
\end{itemize}

While the room for improvement with polynomial curves only (Table~\ref{tab:C1-Ex3-P}) seems to be rather small, using mostly truly rational curve seems to ``over-optimize'' so that the curve energy \eqref{eq:10} no longer is a good measure for the quality of the interpolant. We therefore try a combination of both approaches.

\paragraph{Mixed Approach}

We unite the presented bases of $\mathcal P^{p_1}$ and of $\mathcal R^{p_2}$ for $p_1 = 13$ and $p_2 = 9$ to obtain a new interpolation space $\mathcal Q$. The corresponding solution curve $\mathbf r_o(t)$ is of degree $19$ and has curve energy $E(\mathbf r_o) = 17.0448$. This is lower than what we obtained using polynomial interpolants only. Moreover, Figure~\ref{fig:C1-Ex3-P} confirms that its visual appearance is good while its degree is comparable to the depicted optimal rational interpolant $\mathbf r_{-15}(t)$. These observations essentially also hold true for different combinations of $p_1$ and~$p_2$.

We conclude that in the presented example a mixed polynomial-rational approach seems well-suited for improving a given PH interpolant, provided one is willing to accept a significant raise in the curve degree. This latter drawback does not come as a surprise as our initial curve \eqref{eq:17} (and many other examples from literature) are primarily optimized for low curve degree.

\paragraph{Asymptotic Analysis}

We also performed experimental asymptotic analysis for the $C^1$ interpolation of this example. The data \eqref{eq:16} admits a unique cubic polynomial interpolant $\mathbf{c}(t)$ that is not PH. We use its Bézier form and de Casteljau's algorithm to subdivide it into polynomial cubic curve segments $\mathbf{c}^n_1(t)$, $\mathbf{c}^n_2(t)$, \ldots, $\mathbf{c}^n_{2^n}(t)$. For each curve segment $\mathbf{c}^n_\ell(t)$, we compute a quaternion polynomial $\mathcal{A}^n_\ell(t)$ of degree two according to \cite{sir05} such that the polynomial PH curves $\mathbf{p}^n_\ell(t) \coloneqq \int \mathcal{A}^n_\ell(t) \qi \Cj{\mathcal{A}^n_\ell}(t) \;\mathrm{d}t + \mathbf{p}_0$ exhibit optimal approximation order. In particular,
\begin{equation*}
  \lim_{n = \infty}
  \frac{\max_{\ell, t}\Vert \mathbf{c}^n_\ell(t) - \mathbf{p}^n_\ell(t) \Vert}
       {\max_{\ell, t}\Vert \mathbf{c}^{n+1}_\ell(t) - \mathbf{p}^{n+1}_\ell(t) \Vert}
  = 16 = 2^4
\end{equation*}
so that the optimal approximation order equals four. Using the interpolation spaces $\mathcal{Q}^n_\ell$ spanned by
\begin{equation*}
  \sum_{k=-3}^4 \mathbf{q}_{-1,k}(t+1)^k,\quad
  \sum_{k=-3}^4 \mathbf{q}_{2,k}(t-2)^k,\quad
  \mathbf q_0 = (1, 0, 0),\quad
  \mathbf q_1 = (0, 1, 0),\quad
  \mathbf q_2 = (0, 0, 1),
\end{equation*}
and polynomial solutions of respective degrees $5$, $6$, $7$, and $8$, we run our algorithm with initial solution $\mathbf{p}^n_\ell(t)$ and the curve energy \eqref{eq:10} as optimization target to obtain rational PH curves $\mathbf{r}_\ell^n(t)$. Set
\begin{equation*}
  \varepsilon_n \coloneqq 
  \max_{\ell, t}\Vert \mathbf{c}^n_\ell(t) - \mathbf{r}^n_\ell(t) \Vert.
\end{equation*}
The graph of the ratio $\varepsilon_{n-1}/\varepsilon_n$ for an increasing number $2^n$ of subdivisions is displayed in Figure~\ref{fig:asymptotic-analysis}. Convergence to the optimal value $16$ can be observed. We thus conjecture that optimizing the curve energy does not decrease the approximation order. Using more data from the cubic curve during the optimization (e.g. end point curvatures, middle point interpolation) would be easy and should result in yet a higher approximation order.

\begin{figure}
  \centering
  \includegraphics[scale=0.7]{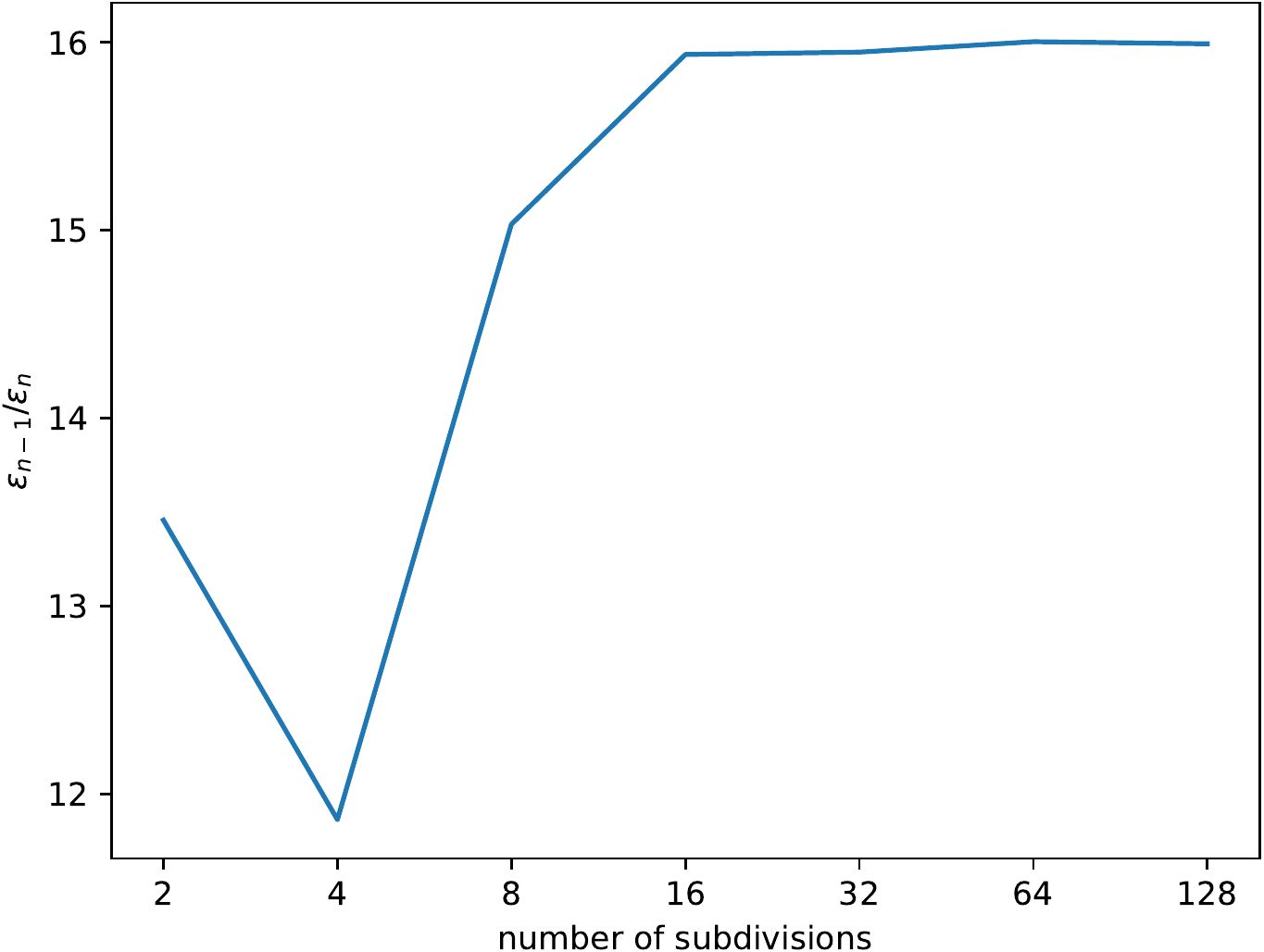}
  \caption{Ratio of approximation errors}
  \label{fig:asymptotic-analysis}
\end{figure}

\subsection*{Complex Roots in the Denominator}

In this section, we consider an example taken from \cite{krajnc3}. There, a rational $G^1$ PH interpolant $\mathbf r(t)$ of class four and degree six is constructed. Once more, it is optimized for low degree so that we cannot expect to be better in that regard but we can improve the curve energy \eqref{eq:10}. The interpolation problem differs from the example discussed in the previous sections in several aspects:
\begin{itemize}
\item We improve upon a truly rational and not upon a polynomial solution.
\item The numerator polynomial has two distinct roots $\beta_1$, $\beta_2$ which both need to be used in the construction of our interpolation space $\mathcal Q$ in order to ensure $\mathbf r(t) \in \mathcal Q$.
\item The thus constructed interpolation space corresponds to a non-generic situation. This needs to be taken into account when computing bases for the subspaces of non-regular and regular curves in~$\mathcal Q$.
\item The two roots $\beta_1$ and $\beta_2$ are conjugate complex so that extra care is required for the construction of a real basis for~$\mathcal Q$.
\end{itemize}

The data to be interpolated is
\begin{equation*}
  \mathbf p_0 = \begin{pmatrix} 0 \\ 0 \\ 0 \end{pmatrix},\quad
  \mathbf p_1 = \begin{pmatrix} 4 \\ 2 \\ 4 \end{pmatrix},\quad
  \mathbf v_0 = \begin{pmatrix} 1 \\ 0 \\ 0 \end{pmatrix},\quad
  \mathbf v_1 = \frac{1}{7} \begin{pmatrix} 6 \\ -3 \\ -2 \end{pmatrix}.
\end{equation*}
The symbolic representation of the interpolant $\mathbf r(t)$ underlying our computation is available but it is too long to be presented here. A numeric approximation is
\begin{equation}
  \label{eq:19}
  \mathbf r(t) =
  \frac{1}{\alpha}
 \begin{pmatrix}
   -1.5063 t^6  + 3.2798 t^5  + 0.6574 t^4  - 0.4547 t^3  - 2.7375 t^2  + 1.8173 t \\
   -0.9345 t^6  + 4.7702 t^5  - 4.3635 t^4  - 0.4475 t^3  + 1.5033 t^2 \\
   0.5818 t^6  + 2.1318 t^5  - 3.0460 t^4  + 0.1859 t^3  + 1.2027 t^2
  \end{pmatrix}
\end{equation}
where $\alpha = (t - \beta_1)^3(t-\beta_2)^3$ and
\begin{equation*}
  \beta_1 = \frac{1}{2} + \frac{\sqrt{6071202867}}{124526} \mathrm i,\quad
  \beta_2 = \frac{1}{2} - \frac{\sqrt{6071202867}}{124526} \mathrm i.
\end{equation*}
Note the complex unit $\mathrm i$ which is to be distinguished from the quaternion unit $\qi$. The rational PH curve \eqref{eq:19} corresponds to
\begin{equation*}
  \mathcal A(t) = (-0.5673 - 0.2579 \qi + 0.5158 \qj - 0.6447 \qk)t^2
                - (1.2721 - 0.3309 \qi + 0.6618 \qj - 0.8272 \qk)t + 1.
\end{equation*}
The polynomial $\mathcal A(t)$ is likewise available in symbolic form.

It turns out that the input data to this example is \emph{non-generic}. This is not as surprising as it may seem at first sight because the construction in \cite{krajnc3} aims for exceptionally low degree of the rational PH interpolant, a property that is related to non-genericity, basically because the constructions of Section~\ref{sec:non-regular} might then allow for the exceptional vanishing of further coefficients in Laurent expansions. More precisely, expanding $\mathbf F(t) \coloneqq \mathcal A(t) \qi \Cj{\mathcal A(t)}$ as
\begin{equation*}
  \mathbf F(t) = \sum_{\ell=0}^4 \mathbf f_{1,\ell}(t-\beta_1)^\ell
               = \sum_{\ell=0}^4 \mathbf f_{2,\ell}(t-\beta_2)^\ell
\end{equation*}
we see that the vectors $\{ \mathbf f_{j,0}, \mathbf f_{j,1}\}$ are linearly dependent as are the vectors $\{\mathbf f_{j,0}, \mathbf f_{j,2}, \mathbf f_{j,3} \}$ for $j \in \{1,2\}$. This means that we are in a case similar to that of Example~\ref{ex:4}. As usual, our interpolation space $\mathcal Q$ will contain the constant solutions $\mathbf{q}_1 = (1,0,0)$, $\mathbf{q}_2 = (0,1,0)$, $\mathbf{q}_3 = (0,0,1)$ but we also add the spaces $\mathcal{N}^{\beta_1}$, $\mathcal{N}^{\beta_2}$ of non-regular solutions which are spanned by
\begin{equation*}
  \mathbf s_{j}(t) = \sum_{\substack{\ell=-3\\\ell \neq 0}}^3 \mathbf s_{j,\ell} (t-\beta_j)^\ell,\quad
  \mathbf t_{j}(t) = \sum_{\substack{\ell=-1\\\ell \neq 0}}^4 \mathbf t_{j,\ell} (t-\beta_j)^\ell,
\end{equation*}
for $j \in \{1,2\}$, respectively. We normalize these solutions by $\mathbf s_{j,-3} = \mathbf t_{j,-1} = \mathbf F(\beta_j)$. These bases of $\mathcal N^{\beta_1}$ and $\mathcal N^{\beta_2}$ are complex but the basis vectors come in conjugate complex pairs. Therefore, $\mathcal N^{\beta_1} \oplus \mathcal{N}^{\beta_2}$ is also spanned by the real basis vectors
\begin{equation*}
  \begin{aligned}
  \mathbf q_4(t) &\coloneqq \tfrac{1}{2}(\mathbf s_1(t) + \mathbf s_2(t)) = \RE(\mathbf s_1(t)) = \RE(\mathbf s_2(t)),\\
  \mathbf q_5(t) &\coloneqq \tfrac{1}{2\mathrm i}(\mathbf s_1(t) - \mathbf s_2(t)) = \IM(\mathbf s_1(t)) = -\IM(\mathbf s_2(t)),\\
  \mathbf q_6(t) &\coloneqq \tfrac{1}{2}(\mathbf t_1(t) + \mathbf t_2(t)) = \RE(\mathbf t_1(t)) = \RE(\mathbf t_2(t)),\\
  \mathbf q_7(t) &\coloneqq \tfrac{1}{2\mathrm i}(\mathbf t_1(t) - \mathbf t_2(t)) = \IM(\mathbf t_1(t)) = -\IM(\mathbf t_2(t)).
  \end{aligned}
\end{equation*}

The real $\mu$-polynomials of these modified basis vectors are obtained from the original $\mu$-polynomials by identical linear combinations. It turns out that the original curve \eqref{eq:19} of \cite{krajnc3} is contained in $\mathbb R \oplus \mathcal N^{\beta_1} \oplus \mathcal N^{\beta_2}$ so that further basis vectors are not required for obtaining an optimal solution. Nonetheless, we augment the thus obtained collection of basis vectors by $0 \ge p \ge 5$ further polynomial basis vectors $\mathbf q_8(t)$, $\mathbf q_9(t)$, \ldots, $\mathbf q_{p+8}(t)$. This results in interpolation spaces of $\dim \mathcal Q = 7+p$. The respective minimizer $\mathbf r_p(t)$ of \eqref{eq:10} has degree $\deg \mathbf r_p(t) = 10+p$.

\begin{table}
  \centering
  \begin{tabular}{c|cccccc}
    \small
    $p$                & $0$      & $1$      & $2$      & $3$      & $4$      & $5$  \\
    \hline
    $\dim \mathcal Q$  & $7$      & $8$      & $9$      & $10$     & $11$     & $12$ \\
    $\deg \mathbf r_p$ & $10$     & $11$     & $12$     & $13$     & $14$     & $15$ \\
    $E(\mathbf r_p)$   & $134.09$ & $131.17$ & $130.16$ & $129.02$ & $128.35$ & $127.99$
  \end{tabular}
  \caption{Dimension of interpolation space $\mathcal Q$, degree of optimal solution $\mathbf r_p$, and curve energy $E(\mathbf r_p)$. The energy of the original curve is approximately $134.10$.}
  \label{tab:G1-Ex2-RP}
\end{table}

\begin{figure}
  \centering
  \includegraphics[]{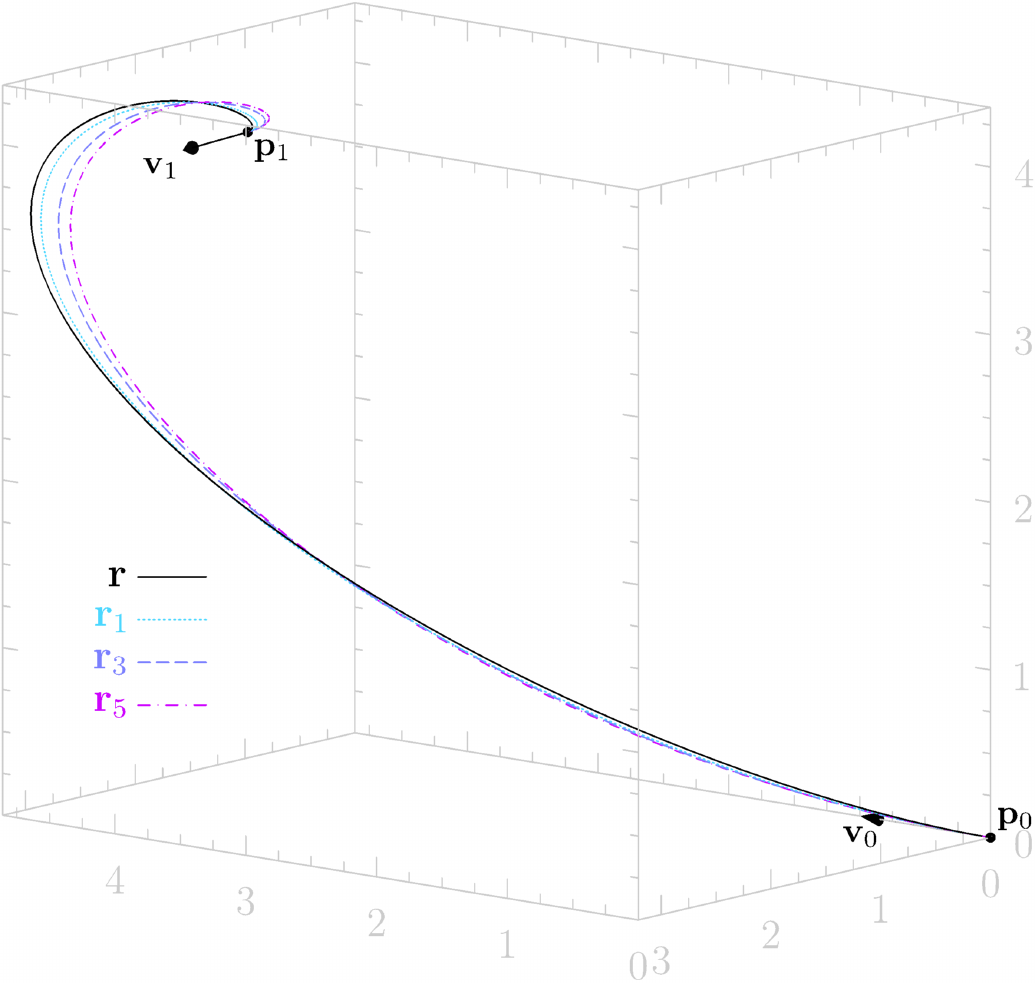}
  \caption{Optimized $G^1$ rational PH interpolants for increasing dimension of interpolation space.}
  \label{fig:G1-Ex2-RP}
\end{figure}

Our findings for optimizing with respect to the curve energy \eqref{eq:10} for various values of $p$ are illustrated in Table~\ref{tab:G1-Ex2-RP} and in Figure~\ref{fig:G1-Ex2-RP}. The energy of the original curve equals $E(\mathbf r) \approx 134.10$. Using only truly rational solutions ($p = 0$) hardly results in an improved curve energy but raises the degree to $10$. Adding further polynomial basis vectors results in a clear improvement of curve energy and visible changes in shape.

\subsection*{Interpolation with Length Constraints}

In our final example we discuss optimization with respect to the length-related functionals \eqref{eq:11} and \eqref{eq:12}. We start with the polynomial quintic PH curve
\begin{equation}
  \label{eq:20}
  \mathbf r(t) =
  \frac{1}{2646000}
  \begin{pmatrix}
    -1864107 t^5  - 1124550 t^4  + 4941300 t^3  - 7371000 t^2  + 2794500 t \\
     370440 t^5  + 6769350 t^4  - 512400 t^3  + 1323000 t^2  - 2268000 t \\
    -2275560 t^5  + 1477350 t^4  - 5419400 t^3  + 7497000 t^2  - 2214000 t
  \end{pmatrix}.
\end{equation}
and consider the $G^1$ interpolation problem with data
\begin{equation*}
  \mathbf p_0 = \mathbf r(0),\quad
  \mathbf p_1 = \mathbf r(1),\quad
  \mathbf v_0 = \mathbf{r'}(0),\quad
  \mathbf v_1 = \mathbf{r'}(1).
\end{equation*}
The curve \eqref{eq:20} is depicted in Figure~\ref{fig:G1-Ex4-Length}. It defines the quaternion polynomial
\begin{equation*}
  \mathcal A(t) =
  (1-\tfrac{1}{5}\qi+2\qj+\tfrac{3}{4}\qk)t^2
  +(\tfrac{1}{3}+3\qi+\tfrac{2}{3}\qj-\qk)t
  -\tfrac{4}{7}-\qi+\tfrac{1}{7}\qj+\tfrac{1}{2}\qk
\end{equation*}
and the vectorial quaternion polynomial $\mathbf F(t) = \mathcal A(t) \qi \Cj{\mathcal A}(t)$.

In contrast to previous examples, we fix the interpolation space $\mathcal Q$ and instead vary the functional to be optimized (and also some constraints). The interpolation space we select is of dimension ten. It is spanned by
\begin{itemize}
\item The three constant solutions $\mathbf q_0 = (1, 0, 0)$, $\mathbf q_1 = (0, 1, 0)$, $\mathbf q_2 = (0, 0, 1)$.
\item Four non-regular rational solutions
  \begin{gather*}
    \mathbf q_4(t) = \sum_{\substack{\ell = -4\\\ell \neq 0}}^3 \mathbf q_{4,\ell} (t-\beta_1)^\ell,\quad
    \mathbf q_5(t) = \sum_{\substack{\ell = -3\\\ell \neq 0}}^4 \mathbf q_{5,\ell} (t-\beta_1)^\ell,\\
    \mathbf q_6(t) = \sum_{\substack{\ell = -4\\\ell \neq 0}}^3 \mathbf q_{6,\ell} (t-\beta_2)^\ell,\quad
    \mathbf q_7(t) = \sum_{\substack{\ell = -3\\\ell \neq 0}}^4 \mathbf q_{7,\ell} (t-\beta_2)^\ell.
  \end{gather*}
  to $\beta_1 = -1$ and $\beta_2 = 2$, respectively. These rational PH curves are normalized by $\mathbf q_4(t) = \mathbf q_5(t) = \mathbf F(\beta_1)$ and $\mathbf q_6(t) = \mathbf q_7(t) = \mathbf F(\beta_2)$, respectively. The corresponding $\mu$-polynomials are
  \begin{gather*}
    \mu_4(t) = (t-\beta_1)^5\lambda_4(t),\quad
    \mu_5(t) = (t-\beta_1)^6\lambda_5(t),\\
    \mu_6(t) = (t-\beta_2)^5\lambda_6(t),\quad
    \mu_7(t) = (t-\beta_2)^6\lambda_7(t),
  \end{gather*}
  respectively.
\item The three polynomial solutions
  \begin{equation*}
    \mathbf q_8(t) = 4\int F(t) \;\mathrm{d}t,\quad
    \mathbf q_9(t) = 5\int t F(t) \;\mathrm{d}t,\quad
    \mathbf q_{10}(t) = 6\int t^2 F(t) \;\mathrm{d}t
  \end{equation*}
  with $\mu$-polynomials $\mu_8(t) = 4$, $\mu_9(t) = 5t$, and $\mu_{10}(t) = 6t^2$, respectively.
\end{itemize}

\begin{figure}
  \centering
  \includegraphics[]{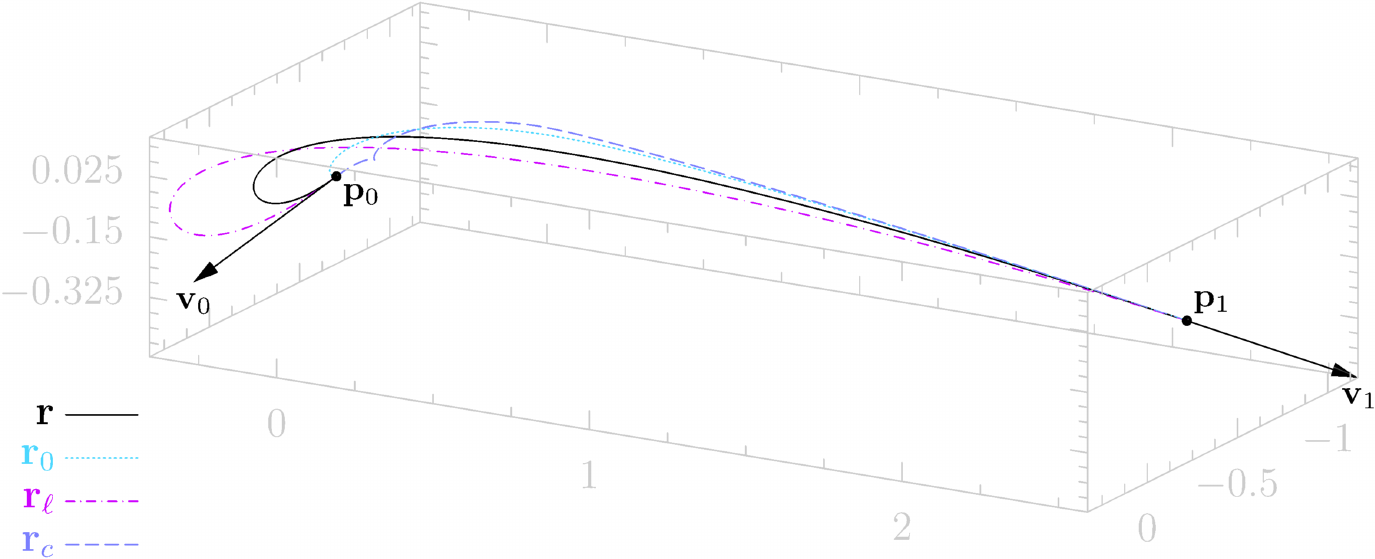}
  \caption{$G^1$ interpolants with length constraints.}
  \label{fig:G1-Ex4-Length}
\end{figure}

The arc-length of the curve $\mathbf r(t)$ in \eqref{eq:20} in the interval $[0,1]$ equals approximately $2.99$. Minimizing the functional \eqref{eq:11} over $\mathcal Q$, we find a solution curve $\mathbf r_0(t)$ of minimal arc-length $s_0 \approx 2.53$. Indeed, it is significantly shorter than the original curve $\mathbf r(t)$ and also differs considerably in shape (Figure~\ref{fig:G1-Ex4-Length}). Using \eqref{eq:12} we can also set a target arc-length $s$. The curve $\mathbf r_\ell$ in Figure~\ref{fig:G1-Ex4-Length} has arc-length $s = 3.5$. We observe that the optimization succeeds as long as $s \ge s_0$, possibly because $\mathcal Q$ contains cuspless interpolating curves of arbitrary arc-length.

Finally, Figure~\ref{fig:G1-Ex4-Length} also illustrates that our non-negativity constraints on the Bernstein coefficients of $\mu(t)$ are essential. Minimizing the arc-length while only requiring the Bernstein coefficients to be larger than $-100$ (we pick some small value in order to avoid an unconstrained optimization problem) results in a solution curve $\mathbf r_c(t)$ of ``minimal arc-length'' which is clearly invalid as it exhibits a cusp near the start point $\mathbf p_0$. Our implementation computes its ``arc length'' as $2.26$ which is smaller than $s_0$. This is the case because our code (on purpose!) does not attempt to correctly resolve the sign when symbolically simplifying $\Vert \mathbf r'(t) \Vert$. Actually, the value $2.26$ is the difference of arc-lengths between cusp and start and end point, respectively.

\section{Conclusion}

We have presented a new, residuum based approach to rational PH curve. It generalizes the well-known construction of polynomial PH curves by direct integration of the hodograph to the rational case and allows for a systematic construction of spaces of PH curves with parallel tangents. While existing interpolation schemes with PH curves focus on optimal (small) curve degrees, our approach is well-suited for straightforward incorporation of curve energy or arclength constraints. We have suggested several interpolation schemes, some of them new and some of them improving upon existing solutions.

In doing so we emphasized the role of the function $\lambda(t)$ in \eqref{eq:3} which provides additional degrees of freedom and computational advantages. In many applications, a polynomial $\lambda(t)$ might be sufficient but rationality has undoubtedly benefits as well. We have in mind PH curves with circularity properties (for example circular projections) which can never be polynomial. Unlike polynomial curves, rational curves can also model bounded closed paths which are useful in kinematics and computer graphics. Our approach with its good control over zeros of the denominator polynomials might be well-suited for designing such curves. In the past, polynomial PH curves have been used for the construction of rational pipe surfaces \cite{sir05}. Rational curves can equally well be used as spine curve, thus simply increasing the space of available surfaces. Since the resulting pipe surfaces are rational anyway, polynomiality of the spine curve feels like an unnecessary restriction.

We would like to emphasize that our results can be extended in several ways:
\begin{itemize}
\item The computations of Section~\ref{sec:residuum} generalize to an ambient space of arbitrary dimension. All we have to do is replace the ``magic number'' three with the dimension $N$ of this space. If $\Span\{\mathbf f_0, \mathbf f_1, \ldots, \mathbf f_n \} = \mathbb R^N$, then the dimension of $\dim \mathcal N^\beta = (n + 1) - N$ (cf. Lemma~\ref{lem:dimN}). Assuming $N > 3$, a canonical basis of $\mathcal N^\beta$ will, for fixed $n$, contain fewer elements (cf. Remark~\ref{rem:canonical-basis}). In particular, it will generically be empty unless the numerator polynomials contains roots of multiplicity at least~$N$.
\item In this paper we constructed the tangent field $\mathbf F(t)$ as $\mathbf F(t) = \mathcal{A}(t) \qi \Cj{\mathcal{A}}(t)$ so that resulting curve is PH. However, the PH property is actually not used at all in the ensuing computations. We can equally well compute Minkowski PH curves \cite{moon99} or any class of rational curves with special properties of the tangent field.
\end{itemize}

In the future, we plan to use similar method to describe the subspace of $\mathcal R$ containing the curves with the rational are length function. Since the tangent vector fields $\mathbf F(t)$ encodes surprisingly many properties of the corresponding PH curves, for example rationality of rotation minimizing frames \cite{FaroukiSir2}, it seems natural to solve related interpolation and approximation tasks using methods explained in this article. Finally, generalizations to rational surfaces with a Pythagorean normal field are conceivable as well but more difficult due to the absence of an advantageous partial fraction decomposition.
 
\bibliographystyle{elsarticle-num}

\end{document}